\documentclass[abstracton]{scrartcl} 

\usepackage{amsfonts}
\usepackage{amsmath}
\usepackage{amssymb}
\usepackage{xcolor}
\usepackage{dsfont}

\usepackage{graphics}
\usepackage{graphicx}
\usepackage{epstopdf}
\usepackage{epsf}
\usepackage[T1]{fontenc}

\usepackage{subcaption}
\usepackage{float}

\usepackage{latexsym}
\usepackage{oldgerm}
\usepackage{dsfont}
\usepackage{amscd, amsthm,enumerate,verbatim,calc}
\usepackage{authblk}
\usepackage{enumitem}
\usepackage{thm-restate}

\newtheorem{theorem}{Theorem}
\newtheorem{Lemma}[theorem]{Lemma}
\newtheorem{Claim}[theorem]{Claim}

\newtheorem{corollary}[theorem]{Corollary}

\newtheorem{observation}[theorem]{Observation}

\newcommand{\w}{\color{black}}

\usepackage[colorlinks,breaklinks,backref]{hyperref}
\usepackage{hyperref}
\usepackage{backref}
\usepackage{todonotes}

\usepackage{thmtools}
\usepackage{thm-restate}
\usepackage{cleveref}

\begin{document}

\title {On $3$-colorability of $(claw, diamond)$-free graphs}

\author[1]{Nadzieja Hodur}
\author[1]{Monika Pil\'sniak}
\author[1]{Magdalena Prorok}
\author[1, 2]{Ingo Schiermeyer}
\affil[1]{\normalsize AGH University of Krakow, al. Mickiewicza 30, 30-059 Krak\'ow, Poland}
\affil[2]{\normalsize TU Bergakademie Freiberg, 09596 Freiberg, Germany}

\date{\today}
 
\maketitle
\begin{abstract}
 The $3$-colorability problem is a well-known NP-complete problem and it remains NP-complete for $(claw, diamond, K_4)$-free graphs. Recently, $3$-colorability has been also considered for $(claw, N_{1,1,1})$-free graphs.
  Here, a generalised net $N_{i, j, k}$ is the graph obtained by identifying each vertex of a triangle with an endvertex of one of three vetex-disjont paths of lengths $i, j, k$. 
  
 We study the class of $(claw, diamond, N_{i, j, k})$-free graphs for $(i, j, k) \in \{(1, 1, 3),$ $ (1, 2, 2), $ $(2, 2, 2) \}$. We show that these graphs are $3$-colorable or contain a $K_4$ or belong to some well-defined class of non $3$-colorable graphs. Moreover, we prove that there are only finitely many non $3$-colorable $N_{1, 2, k}$-free graphs for any $k \geq 2$, but there exist infinitely many non $3$-colorable $N_{i, j, k}$-free graphs for any 
 $2 \leq i \leq j \leq k.$
\end{abstract}

\noindent
{\w Keywords: $4$-colorable graphs, forbidden induced subgraphs, perfect graphs, complexity}            
Math. Subj. Class.: { 05C15,  05C17, 68Q25, 68W40. }                      

\section{Motivation and Introduction}\label{sec:intro}

We consider finite, simple, and  undirected graphs.
For terminology and notations not defined here, we refer to~\cite{BM08}. In this paper we will use the following notation. A \textit{diamond} is a graph $K_4$ with one edge deleted; a \textit{bull} is a triangle with two pendant edges attached to two of the vertices of the triangle; and a \textit{net} is a triangle with additional pendant edge at each vertex. A \textit{generalised net} $N_{i, j, k}$ has its pendant edges subdivided $i - 1, j-1, k-1$ times respectively (see Fig. \ref{pic:net}.)

    \begin{figure}[htb]
      \begin{subfigure}{0.3\textwidth}
    \centering  
    \includegraphics[width=0.8\linewidth]{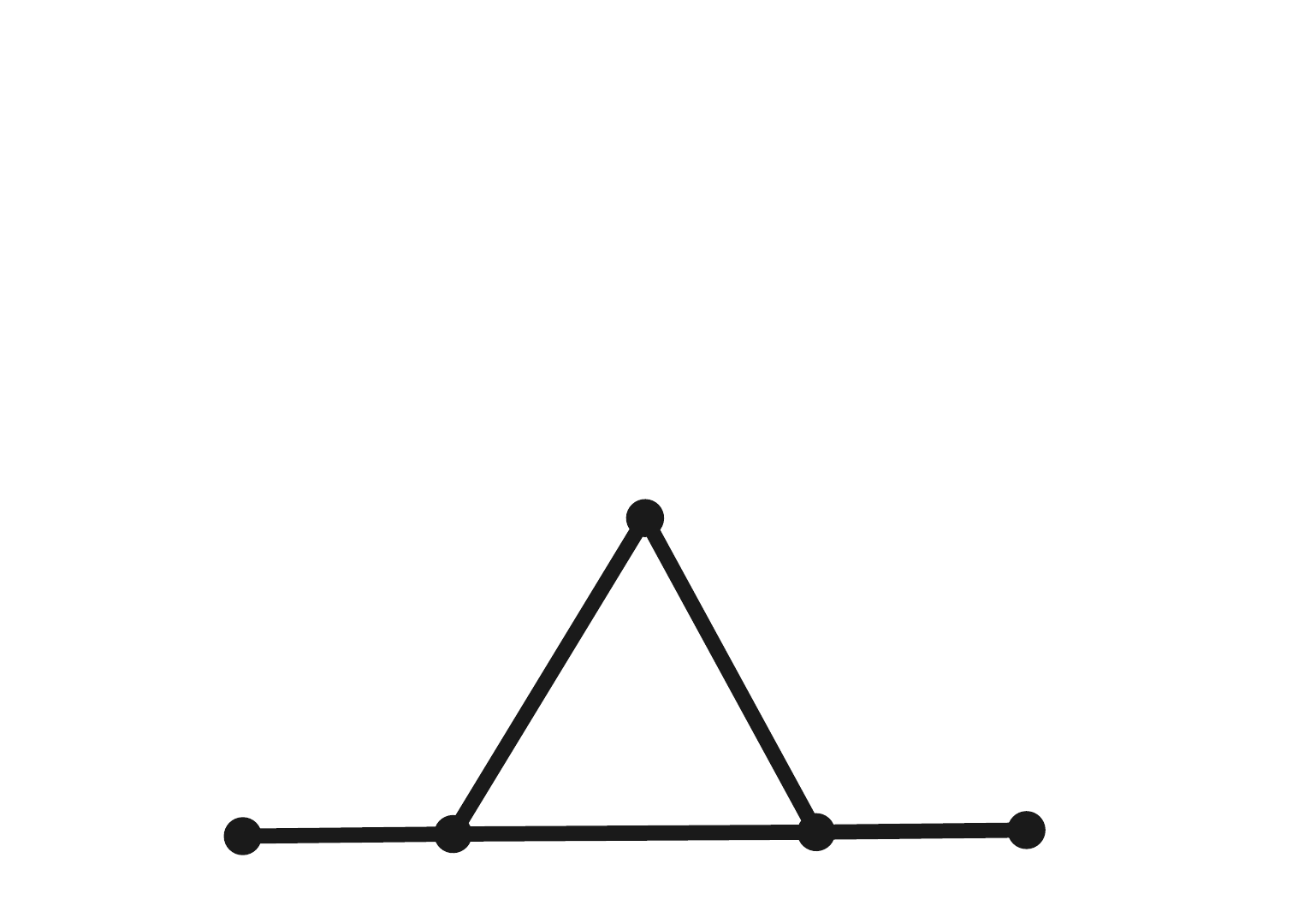}
    \caption{\label{pic3}}
    \end{subfigure}
    \centering    
    \begin{subfigure}{0.3\textwidth}
    \centering  
    \includegraphics[width=0.8\linewidth]{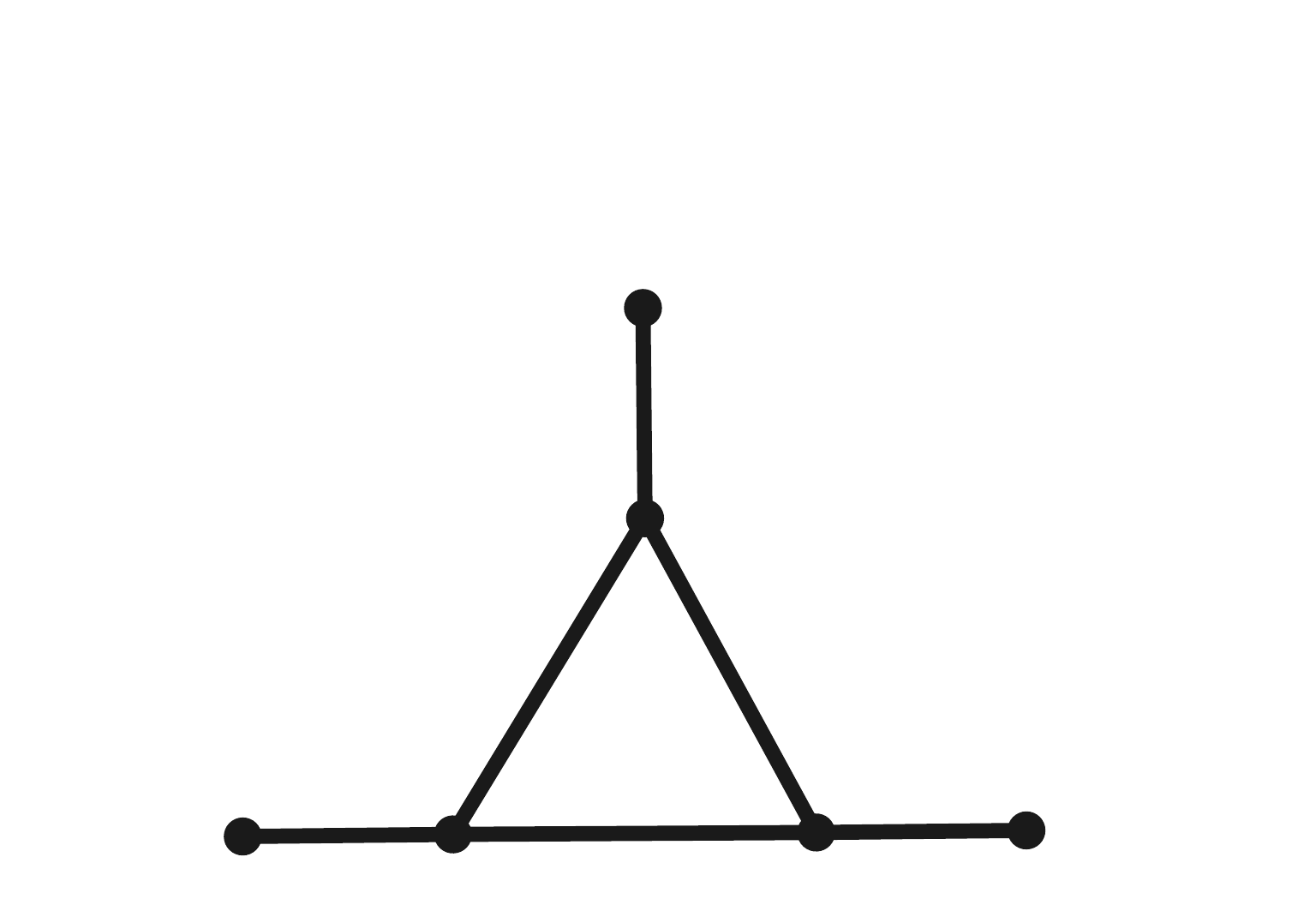} 
    \caption{\label{pic1}}
    \end{subfigure}
    \begin{subfigure}{0.3\textwidth}
    \centering  
    \includegraphics[width=0.8\linewidth]{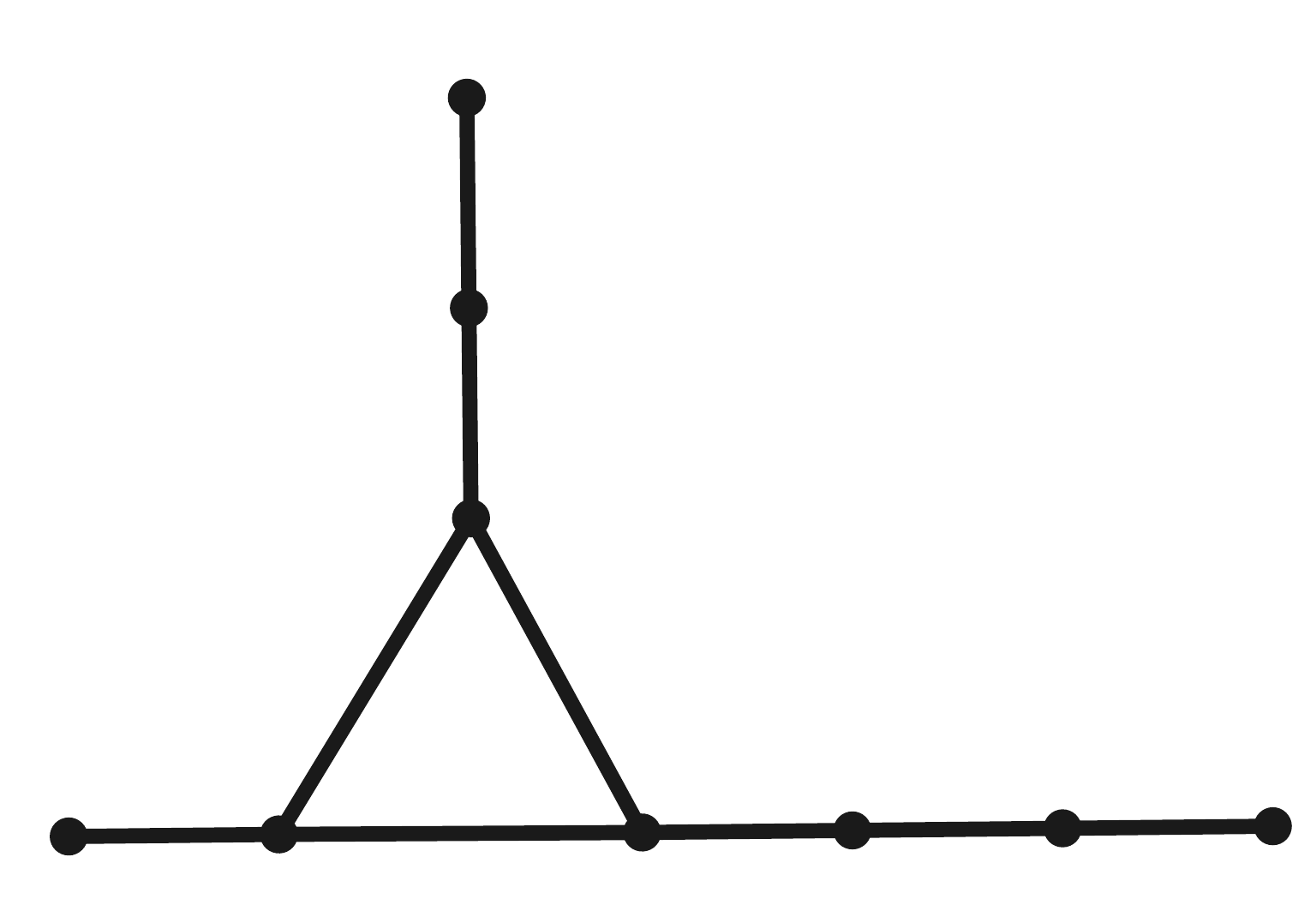}
    \caption{\label{pic2}}
    \end{subfigure}
  
    \caption{The graphs called the bull, the net and $N_{1,2,3}$ respectively\label{pic:net}.}
\end{figure} 

It is well-known that the $k$-colorability problem remains NP-complete in classes of graphs defined by one forbidden induced subgraph $F$, unless this subgraph is a linear forest. 
{If $F$ is not a linear forest then $k$-Colorability remains $NP$-complete. The cases when $F$ is a triangle or a claw have been studied very intensively. 

The $3$-colorability problem was also stated for classes defined by pairs of forbidden induced subgraphs. In 2004, Randerath obtained structural results, listing all possible pairs of forbidden induced subgraphs, which can be saturated pairs implying $3$-colorability of a graph (\cite{Ran1}). There are also many algorithmic results, providing polynomial algorithms deciding whether or not a graph with a forbidden pair is $k$-colorable. For instance, for pairs $(claw, P_6)$, $(claw, K_1 + 2K_2)$  (\cite{Ran}), and quite recently for pairs $(bull, S_{1,1,2})$ and $(bull, S_{1,2,2})$ (\cite{kkol}, \cite{stara}). Here, $S_{i,j,k}$ is a tree consisting  three paths of length ${i, j, k}$ joining by identify one end-vertex of every path. In particular a $claw$ is a graph $S_{1,1,1}$. Moreover, given a family $\cal{H}$ of graphs and a graph $G$, we say that $G$ is \emph{$\cal{H}$-free}
if $G$ contains no graph from $\cal{H}$ as an induced subgraph. In this context, the graphs of $\cal{H}$ are referred to as \emph{forbidden induced subgraphs}. Interestingly, the thorough structural analysis given by \cite{stara} provides that the only $(bull, S_{1,1,2})$-free (and also $(bull, S_{1,2,2})$-free) non-$3$-colorable graphs are build of diamonds, while in this paper we will see that the main families of $(claw, diamond)$-free graphs contain $bull$.

For more results concerning complexity of coloring problems in classes defined by forbidden induced subgraphs, see \cite{kral} and surveys \cite{surv}, \cite{RS04}.

Our research has been motivated by a recent result about $3$-colorability of $(claw,net)$-free graphs \cite{clawdiamondnet}. First observe a necessary and useful fact for $3$-colorable graphs: 

Let $G$ be a $3$-colorable graph containing an induced diamond. Then in any $3$-coloring of $G$ the two nonadjacent vertices in an induced diamond receive the same color.

Based on this observation in \cite{clawdiamondnet} the class of so called {\it spindle graphs} (cf. also \cite{stara}) has been introduced, which all contain induced diamonds and which are all non $3$-colorable.

\begin{theorem} \label{thm::N111} \cite{clawdiamondnet}
	Let $G$ be a connected, $(claw, net)$-free graph. Then 
	\begin{enumerate}
		\item $G$ contains $K_4$ or
		\item $G$ contains $W_5$ or
        \item $G$ contains a spindle subgraph or
		\item $G$ is $3$-colorable. 
	\end{enumerate}
\end{theorem}

Hence, if we forbid additionally diamonds as induced subgraphs, we obtain the following corollary.

\begin{corollary} \label{cor::N111}
	Let $G$ be a connected, $(claw, diamond, net)$-free graph. Then 
	\begin{enumerate}
		\item $G$ contains $K_4$ or
		\item $G$ is $3$-colorable. 
	\end{enumerate}
\end{corollary}

Since spindle graphs have turned out to be a useful certificate for a graph not to be $3$-colorable, a natural question arises:

\medskip

\textit{If a graph $G$ does not contain induced diamonds, what structural arguments can be given to describe that 
it is not $3$-colorable?}

\medskip

{The $3$-Colorability problem remains $NP$-complete for $(claw, diamond, K_4)$-free graphs (\cite{kral}).}
{It has been also shown by \cite{LozPur} that $3$-colorability remains $NP$-complete for $(claw, diamond,$ $K_4,$ $C_4, C_5, \ldots, C_k)$-free graphs for any $k \geq 4$. Using Beineke's Theorem, which characterizes line-graphs in terms of nine forbidden induced subgraphs (cf. [BoM08]), it follows that 
$(claw, diamond, K_4)$-free graphs are line graphs of subcubic triangle-free graphs. Moreover, in graph coloring theory the following  relationship for graphs and their line-graphs is well-known:}

\begin{observation}
Let $H$ be a graph and $G = L(H)$ be its line graph. Then for every integer $k \geq 1$ it holds: $G$ is $k$-colorable if and only if $H$ is $k$-edge colorable.
\end{observation}

The class we are dealing with in this paper has been already studied by Munaro (\cite{linegraphs}) resulting with theorems concerning independent sets and Hamiltonian problems. Our main motivation is an observation used by  Munaro, providing that $(claw, diamond)$-free graphs are line graphs of subcubic triangle-free graphs.

\begin{theorem}\cite{linegraphs}\label{linegraphs}
    The following statements are equivalent, for any graph $G$:
    \begin{enumerate}
        \item $G$ is a $(K_4, claw, diamond)$-free graph.
        \item $G$ is a line graph of a subcubic triangle-free graph.
        \item $G$ is a $(1, 1)$-interval graph such that there do not exist four vertices sharing the same interval and there do not exist three vertices such that any two of them share a different interval.
    \end{enumerate}
\end{theorem}

Thus, instead of $3$-colorability of $(claw, diamond)$-free graphs we will be working with an equivalent problem, that is, $3$-edge-colorability of subcubic triangle-free graphs. Note that a subproblem of the last problem is recognizing
 {\it snarks}, that is,  simple, connected, bridgeless cubic graphs with girth at least five and chromatic index equal to four. 

It is true, since in the class of $3$-regular graphs contracting triangles and diamonds does not change edge-coloring properties. Not surprisingly, it was shown by Král’, Kratochvíl, Tuza, and Woeginger, that in the class considered by Munaro the $3$-Colorability problem is NP-complete (\cite{kral}). But which part of the class makes this problem hard? We are trying to answer this question, considering several subclasses. Interestingly, it seems that all hard-to-deal $(claw, diamond)$-free graphs contain an induced generalised \textit{net} $N_{i, j, k}$.  We will see that after excluding induced net, it is possible to easily determine, whether or not a graph of the considered class is $3$-colorable. In the Section 2.1 we will state and prove that if a $(claw, diamond)$-free graph does not contain an induced $N_{1,1,3}$, then it contains the complete graph $K_4$ or the exceptional graph $B_{10}$ (See Figure \ref{pic:L(B_1)}), or it is $3$-colorable. In Sections 2.2 and 2.3 similar theorems will be stated for $(claw, diamond, N_{2, 2, 2})$-free graphs. Here, the list of exceptional graphs depends on whether or not a graph contains an induced cycle $C_5$. In Section 3 we gather conclusions and ask some further questions.

\hspace{1cm}




Let us consider a connected, $(claw, diamond)$-free graph $G$. In the proofs of our theorems we will also
 assume without loss of generality that $G$ is $K_4$-free. Then, by Theorem \ref{linegraphs} we have that $G = L(H)$, where $H$ is a subcubic triangle-free graph. Of course, a proper edge-coloring of $H$ corresponds to proper vertex-coloring of $G$. Thus, we will work with simple, triangle-free graphs $H$ with maximum degree bounded by $3$, and we will look for their proper edge-colorings. We will apply the following useful reductions for edge-colorings to these graphs.  
 
 \hspace{0,3cm}
\begin{enumerate}
	\item []\label{red1} \textbf{Reduction 1.} If there is a vertex $v \in V(H)$ such that $d(v) = 1$, then $H$ is 3-edge-colorable if and only if $H - v$ is 3-edge-colorable. So we can assume $\delta(H)\geq 2$.

	\item []\label{red2} \textbf{Reduction 2.} If we have two adjacent vertices $u, v \in V(H)$ such that $d(u)=d(v)=2$, then $H$ is 3-edge-colorable if and only if $H - u - v$ is 3-edge-colorable. So we can assume that no two vertices of degree $2$ are adjacent. 

	\item []\label{red3} \textbf{Reduction 3.} Let $C=w_1w_2w_3w_4$ be an induced $4$-cycle in $H$, where $d(w_1)=d(w_3) = 2$. Then $H$ is 3-colorable if and only if $H - \{w_1, w_2, w_3, w_4\}$ is 3-edge-colorable. So we can assume $H$ has no induced $4$-cycle with two independent vertices of degree~$2$.

    \item []\label{red4} \textbf{Reduction 4.} $H$ does not contain cut-edges. Otherwise its line graph $G=L(H)$ would contain a cut-vertex and could be reduced.
\end{enumerate}

 \hspace{0,3cm}
 
The following easy observations characterize graphs which are necessarily non-$3$-edge-colorable. Let us recall that the \textit{matching number} $\alpha'(H)$ of the graph $H$ is a maximal cardinality of a set of independent edges in $H$. We say that $H$ is {\it overfull}, if $|E(H)| > \alpha'(H)\Delta(H)$.

\begin{observation}\label{easyfact}
	If a graph $H$ is overfull, then $H$ is of Vizing class 2, so $\chi'(H)= \Delta(H)+1$.	
\end{observation}

\begin{observation}
	Let $H$ be a cubic bipartite graph and $e \in E(H)$. Let $H^*(e)$ be a graph obtained from $H$ by subdividing the edge $e$. Then $\chi'(H^*(e))=4$.
\end{observation}	
\begin{proof}
	Since $H$ is cubic, we have $|V(H)|=2k$ for some $k$. Moreover, $\alpha'(H) = k$ and we have $\chi'(H)=3$. For $H^*(e)$ we have also $\alpha'(H^*(e))=k$, but $|E(H)| = 3k+1$ and by Observation \ref{easyfact} we have $\chi(H^*(e))=4$.
\end{proof}

A simple example of a non-$3$-edge-colorable graph obtained this way is $K_{3,3}^*$ - complete bipartite graph $K_{3,3}$ with one edge (call it $e$) subdivided. Trying to $3$-color edges of the graph $K_{3,3}^*$ we obtain that both "halves" of $e$ must receive the same color. We can use this observation to obtain a family $D_{6i+1}$ of non-$3$-colorable graphs, joining $i$ copies of $K_{3,3}$ as depicted in Figure \ref{pic:thm::N222-pic1}. In a similar way (see Figure \ref{pic:thm::N222-pic2}) we obtain a family $D_{6i+5}$. Line graphs (which are obviously non-$3$-vertex-colorable) of those families are called $B_{9i+1}$ and $B_{9i+7}$, respectively - see Figures \ref{pic:L(B_1)} and \ref{pic:L(B7)}. 

\begin{figure}[htb]
    \centering    
    \begin{subfigure}{0.4\textwidth}
    \centering  
    \includegraphics[width=1\linewidth]{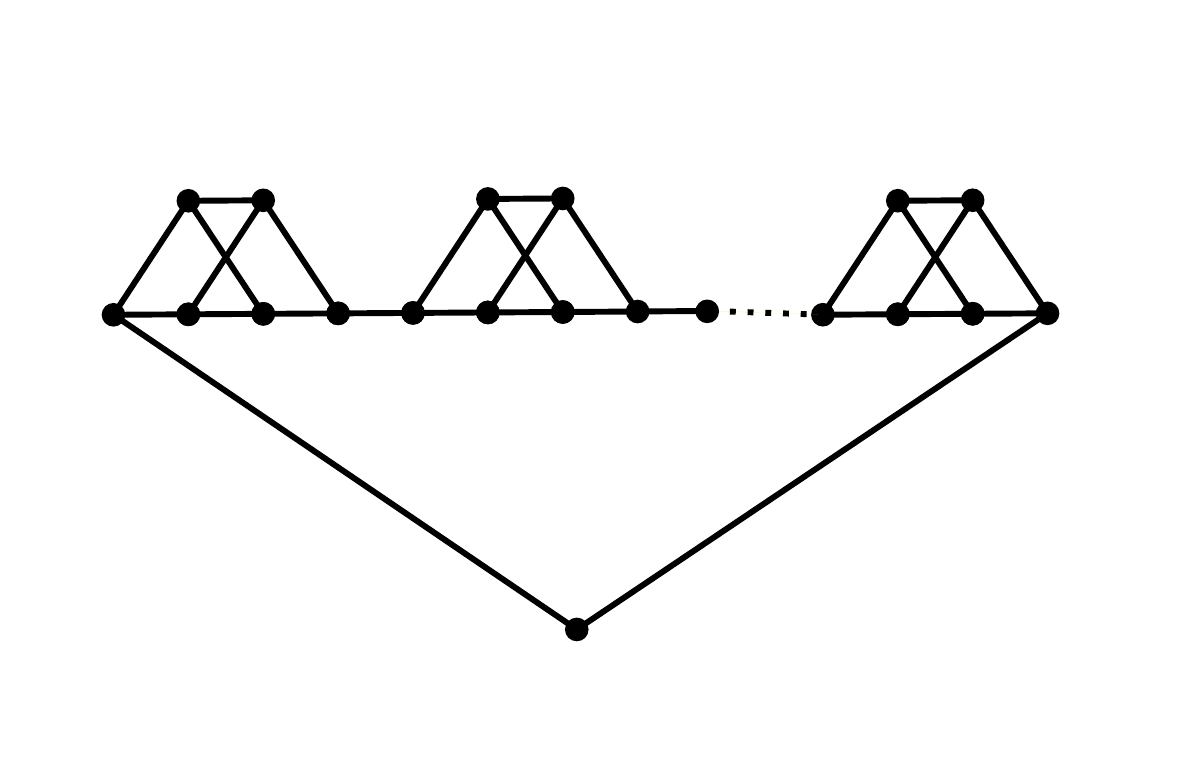} 
    \caption{$D_{6i+1}$ \label{pic:thm::N222-pic1}}
    \end{subfigure}
    \hspace{1cm}
    \begin{subfigure}{0.4\textwidth}
    \centering  
    \includegraphics[width=1\linewidth]{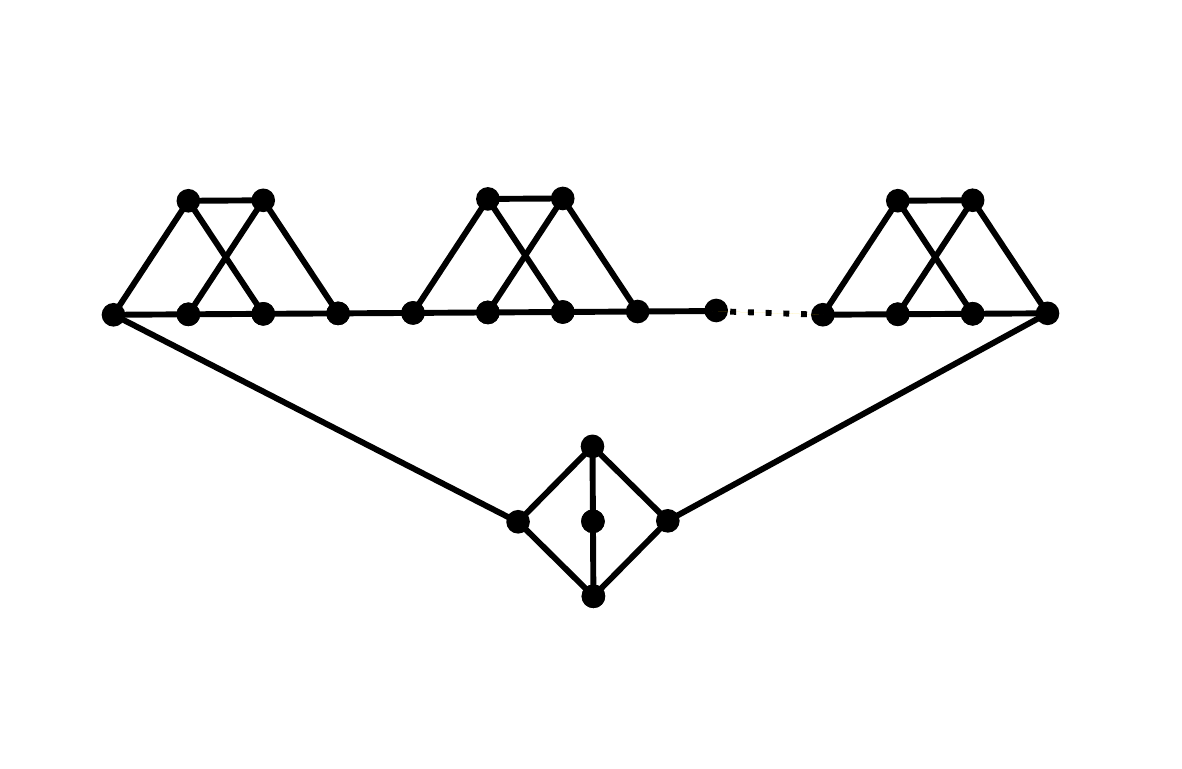}
    \caption{$D_{6i+5}$\label{pic:thm::N222-pic2}}
    \end{subfigure}

    \centering    
    \begin{subfigure}{0.4\textwidth}
    \centering  
    \includegraphics[width=1\linewidth]{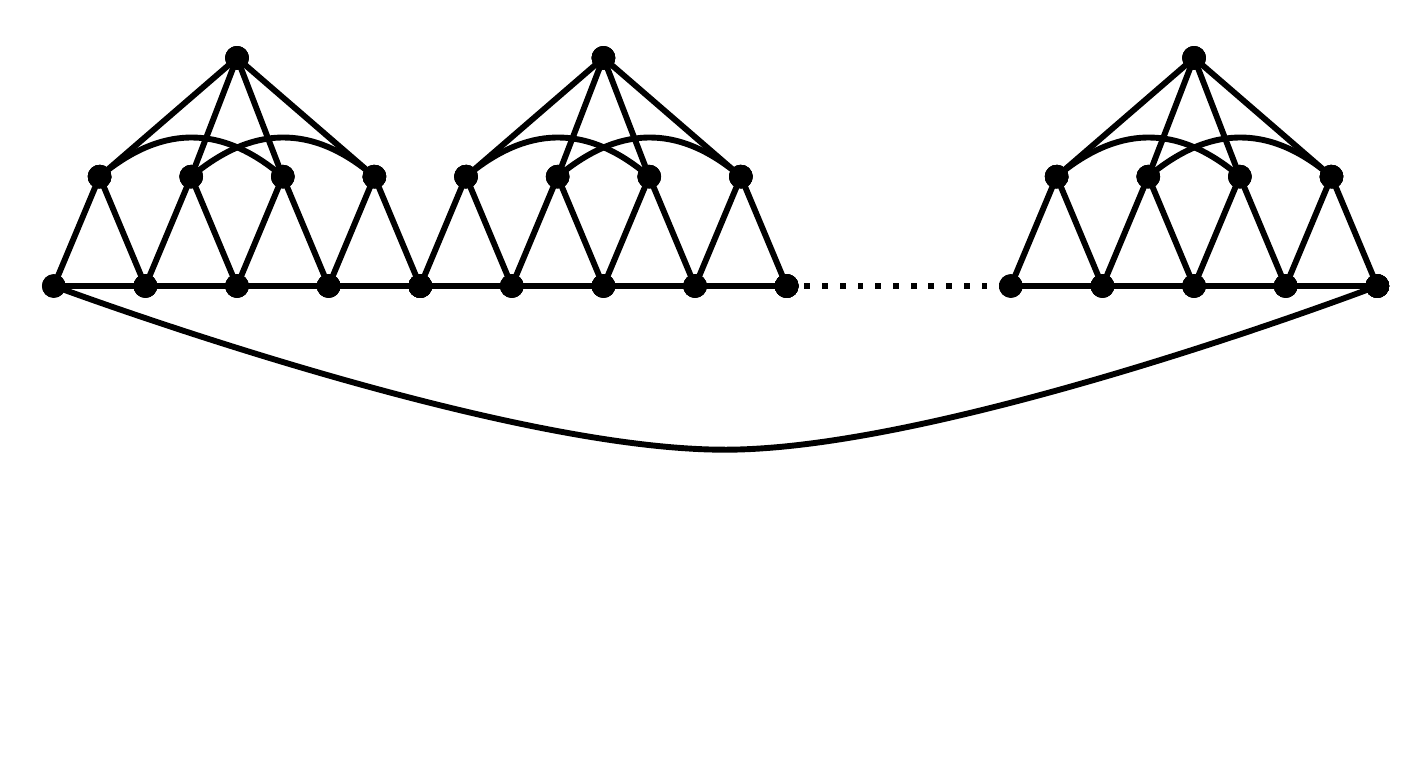} 
    \caption{$B_{9i+1}$\label{pic:L(B_1)}}
    \end{subfigure}
    \hspace{1cm}
    \begin{subfigure}{0.4\textwidth}
    \centering  
    \includegraphics[width=1\linewidth]{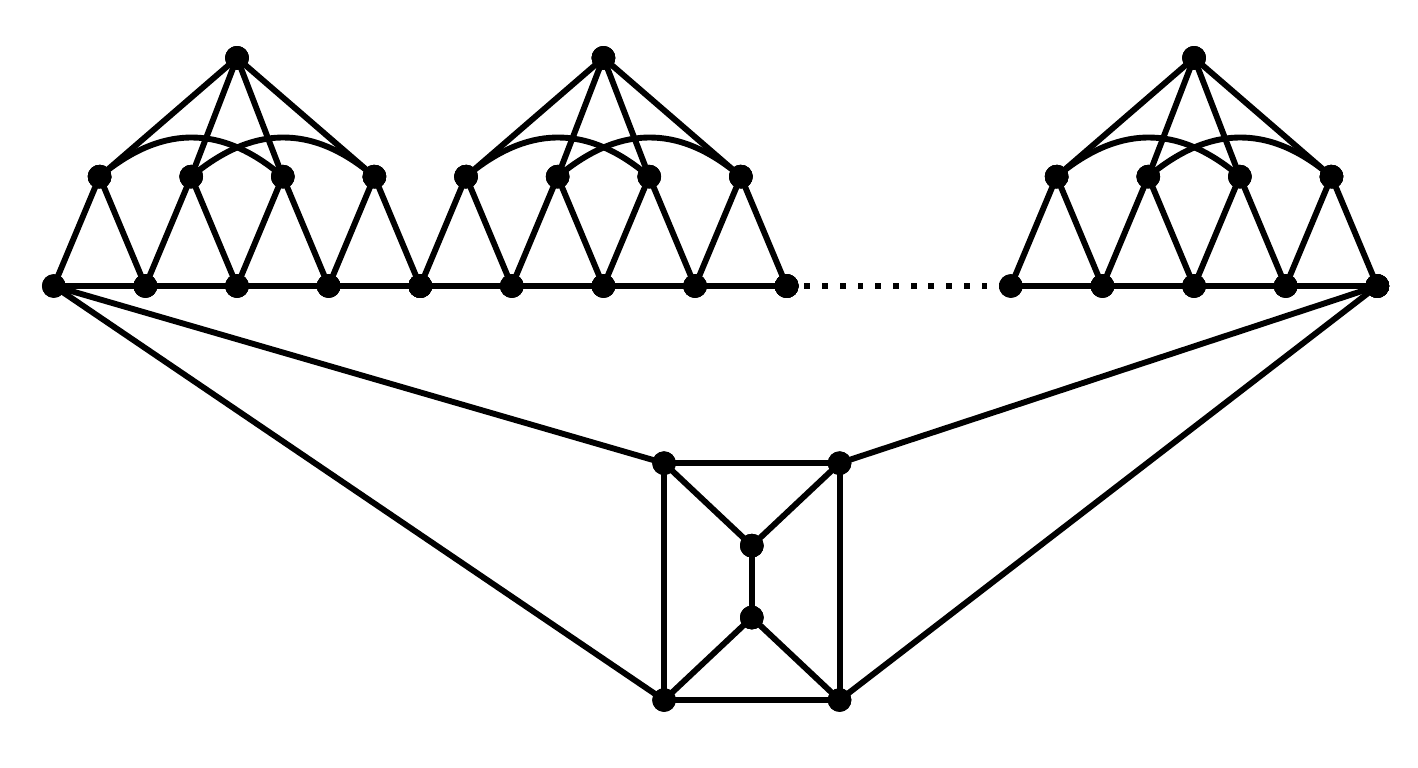}
    \caption{$B_{9i+7}$\label{pic:L(B7)}}
    \end{subfigure}
    \caption{The exceptional non-$3$-colorable, $(claw, diamond)$-free graphs and their base graphs.}
 
\end{figure}

\section{Main results}

\subsection{The (claw, diamond, N$_{1, 1, 3}$)-free graphs}

We use the name \textit{exceptional subgraph} for graphs listed in statements of the theorems, satisfying the following condition: if a graph $G$ of the considered class does not contain any of those subgraphs, then $G$ is colorable with desired number $k$ of colors. Of course, the clique $K_k$ is always an exceptional subgraph. In this section we will show that for $(claw, diamond,  N_{1,1,3})$-free class there exists only one non-trivial exceptional subgraph. 

\begin{theorem} \label{thm::N113}
	Let $G$ be a connected, $(claw, diamond)$-free graph. If $G$ is $N_{1,1,3}$-free, then 
	\begin{enumerate}
		\item $G$ contains $K_4$ or
		\item $G$ contains $B_{10}$ or
		\item $G$ is $3$-colorable. 
	\end{enumerate}
\end{theorem}

\begin{proof}
	Let us assume that $G$ does not contain $K_4$. Then $G=L(H)$, where $H$ is some connected, subcubic, triangle-free graph, which does not contain a (not necessarily induced) $S_{2,2,4}$. Let us apply to $H$ all the Reductions 1-4. We will show that 
	\begin{itemize}
		\item $H = K_{3,3}^* = D_7$ or
		\item $H$ is $3$-edge-colorable. 
	\end{itemize}
	If $\Delta(H) \leq 2$ or $H$ is bipartite, then $\chi'(H) \leq 3$. Hence, we can assume without loss of generality that we can find  in $G$ an induced odd cycle $Q=v_1 \ldots v_p$, where $p \geq 5$. Let us assume that $p$ is minimal. We will distinguish two cases. 

    \vspace{0.5cm}
    \textbf{Case 1: $p\geq 7$}
    
    Note that if $N_2(Q) \neq \emptyset$ or if there is an edge non-adjacent to the cycle, then $H$ contains an (induced or not) $S_{2,2,4}$. Thus, we conclude that $Q$ is dominating and any vertex $v\in N(Q)$ has at least $2$ neighbors on $Q$. Now note that if the vertex $v$ has two neighbors $v_i, v_j \in Q$ such that $|i-j| \neq 2$, then we can find a shorter induced odd cycle, contradicting the minimality of $p$. Thus, for any $v \in V(G)\setminus Q$ we have $N_Q(v) = \{v_i, v_{i+2}\}$ for some $i$. But if so, then by Reduction 3 one of vertices $v, v_{i+1}$ must have an additional neighbor. From the previous steps we know that this neighbor does not belong to the cycle. So if $N(Q)$ is non-empty, we can find a star $S_{2,2,4}$ (see Figure \ref{pic:thm::s224_1}).

\begin{figure}[htb]
    \centering      
    \includegraphics[width=0.4\linewidth]{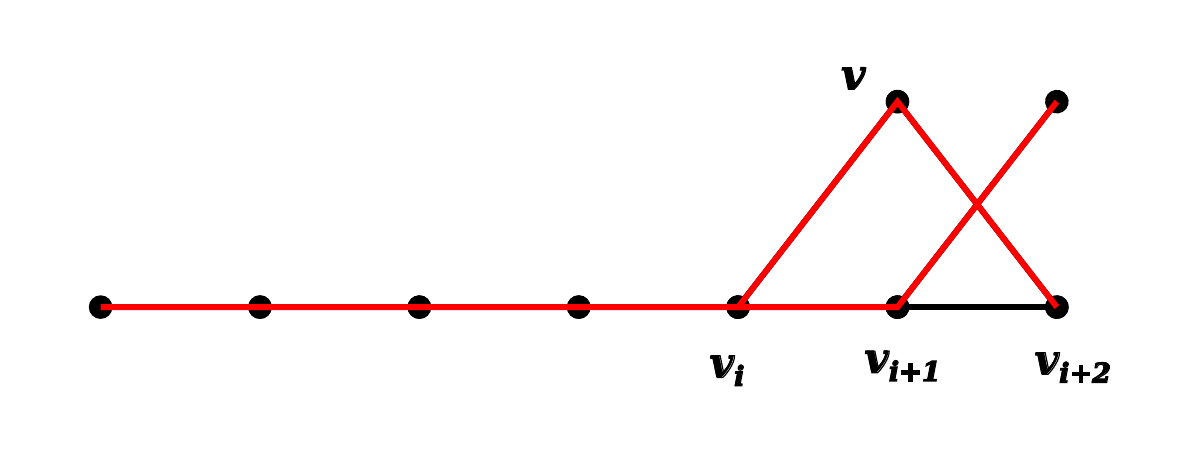} 
    \caption{\label{pic:thm::s224_1} Forbidden subgraph in Case 1 in the proof of Theorem \ref{thm::N113}.}
    
\end{figure}

Thus, in this case the graph $H$ is an odd cycle without chords, so it is $3$-edge-colorable. 

\vspace{0.5cm}
\textbf{Case 2: $p=5$}

Note that if there exists a vertex $v \in N_3(Q)$ adjacent to some $u \in N_2(Q)$, then by Reduction 1 the vertex $v$ must have an additional neighbor $u' \neq u$ and we can find a star $S_{2,2,4}$ (see Figure \ref{pic:thm::s224_2}). Thus, $N_3(Q)= \emptyset$. Similarly, if there exists an edge $uu'$, where $u, u' \in G[N_2]$, then from triangle-freeness of $H$ the vertices $u, u'$ must have two different neighbors $v, v' \in N_1(Q)$ and we can find $S_{2,2,4}$ (see Figure \ref{pic:thm::s224_3}). Thus, $N_2(Q)$ is a set of isolated vertices.  
\begin{figure}[htb]
    \centering    
    \begin{subfigure}{0.4\textwidth}
    \centering  
    \includegraphics[width=1.1\linewidth]{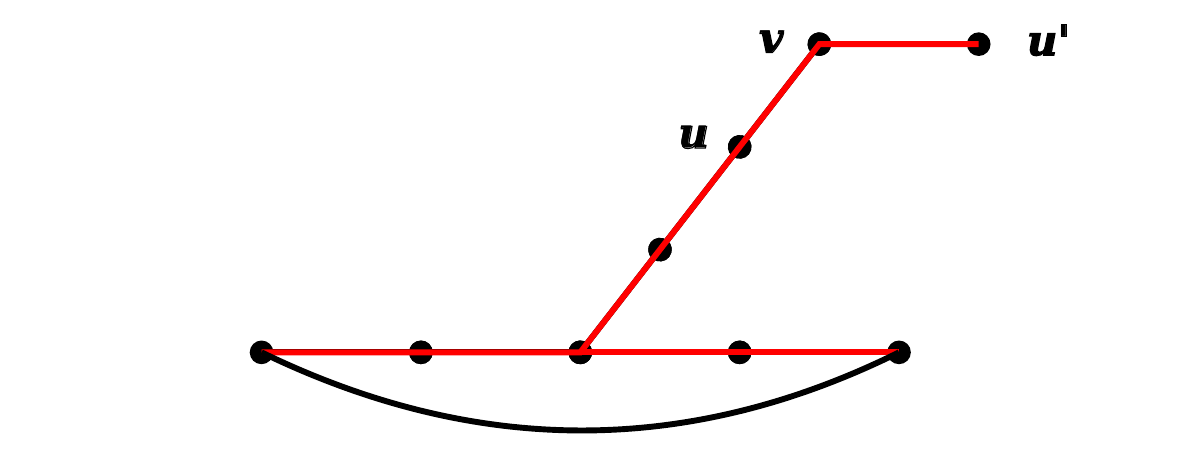} 
    \caption{\label{pic:thm::s224_2}}
    \end{subfigure}
    \begin{subfigure}{0.4\textwidth}
    \centering  
    \includegraphics[width=1.1\linewidth]{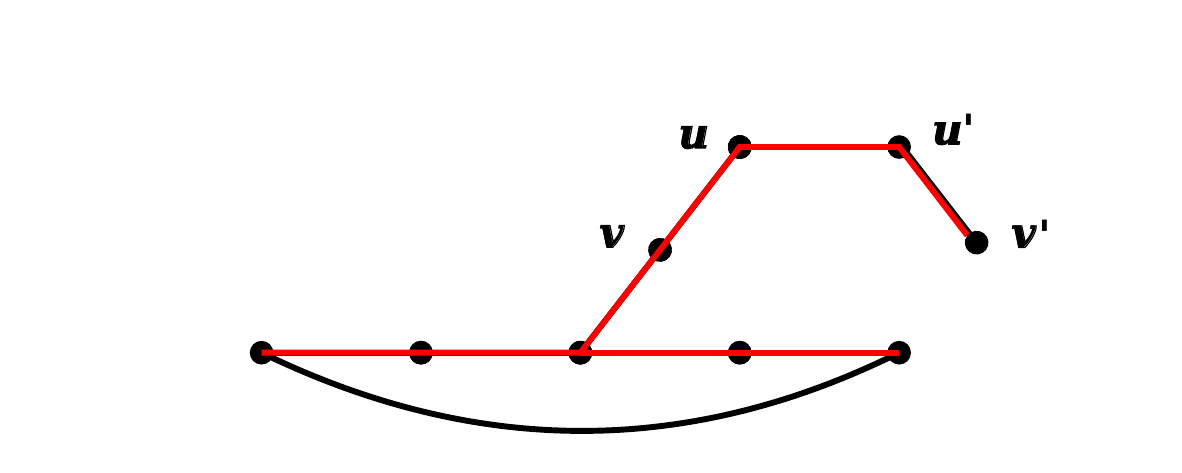}
    \caption{\label{pic:thm::s224_3}}
    \end{subfigure}
    \caption{Forbidden subgraphs in Case 2 of Theorem \ref{thm::N113}.}
\end{figure}

Let us consider a vertex $w \in N_2(Q)$. By Reduction 1 the vertex $w$ has at least two neighbors $v, v'$ (where $v, v' \in N_1(Q)$). By Reduction 2 either $d(w)=3$ or $d(v)=d(v')=3$. 

If $d(w)=3$, then $N(w) = \{u, u', u''\} \subset N_1(Q)$ and any vertex of $N(w)$ has its private neighbor on $Q$ (since $\Delta(G) <4$). If the three private neighbors are consecutive, we can easily find $S_{2, 2, 4}$ (see Figure \ref{pic:thm::s224_4}), so suppose otherwise. If $G = Q \cup \{w, u, u', u''\}$, then $H$ is easily $3$-edge-colorable (see Figure \ref{pic:thm::s224_5}). But if there is more vertices, then there must be a vertex $z$ adjacent either to some vertex $v_i$ of $Q$ (where $v_i$ is not a neighbor of $u, u', u''$) or to some vertex of $\{u, u', u''\}$. In both cases we can find $S_{2, 2, 4}$ (see Figure \ref{pic:thm::s224_6}).

\begin{figure}[htb]
    \centering    
    \begin{subfigure}{0.3\textwidth}
    \centering  
    \includegraphics[width=0.8\linewidth]{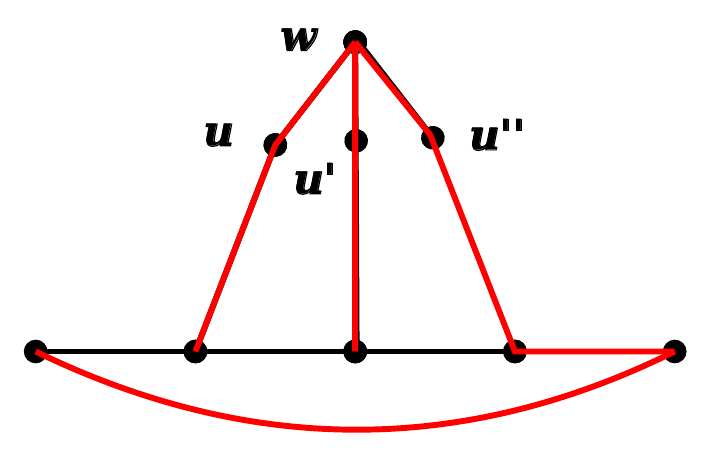} 
    \caption{\label{pic:thm::s224_4}}
    \end{subfigure}
    \begin{subfigure}{0.3\textwidth}
    \centering  
    \includegraphics[width=0.8\linewidth]{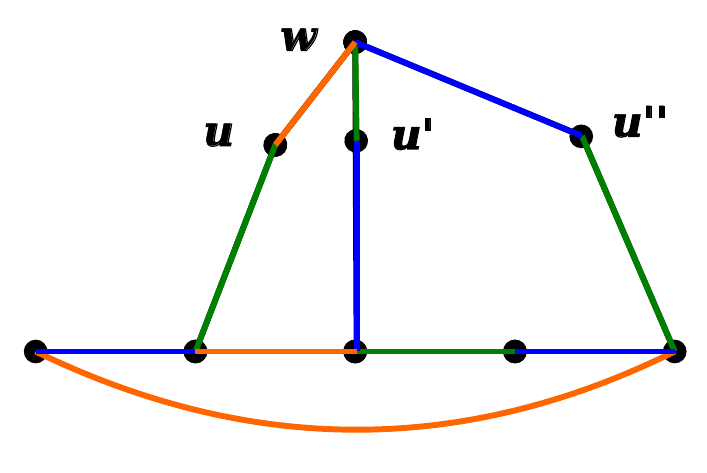}
    \caption{\label{pic:thm::s224_5}}
    \end{subfigure}
    \begin{subfigure}{0.3\textwidth}
    \centering  
    \includegraphics[width=0.8\linewidth]{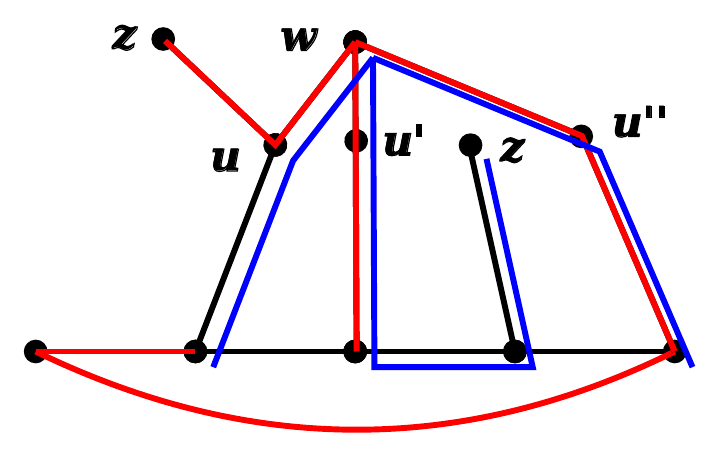}
    \caption{\label{pic:thm::s224_6}}
    \end{subfigure}
    \caption{Forbidden subgraphs and coloring of neighborhood of $C_5$ in Case 2 of Theorem~\ref{thm::N113}.}
\end{figure}    

Hence, we may assume $d(w) = 2$ and $d(u)=d(u')=3$, that is, $u, u'$ have additional neighbors $v$ and $v'$ respectively. If $v \notin Q$ or $v'\notin Q$, we can find $S_{2, 2, 4}$ (see Figure \ref{pic:thm::s224_7}). Otherwise we have $G= Q \cup \{w, u', u''\}$ and $G$ is $3$-edge-colorable (see Figure \ref{pic:thm::s224_8}). 

\begin{figure}[htb]
    \centering    
    \begin{subfigure}{0.3\textwidth} 
    \centering  
    \includegraphics[width=0.8\linewidth]{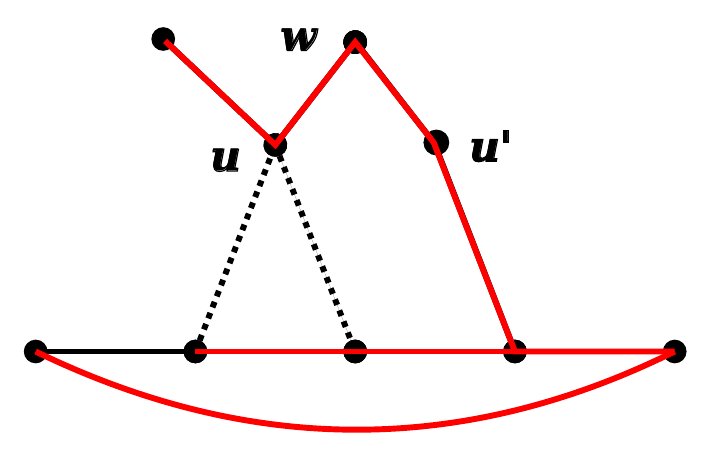} 
    \caption{\label{pic:thm::s224_7}}
    \end{subfigure}
    \begin{subfigure}{0.3\textwidth}
    \centering  
    \includegraphics[width=0.8\linewidth]{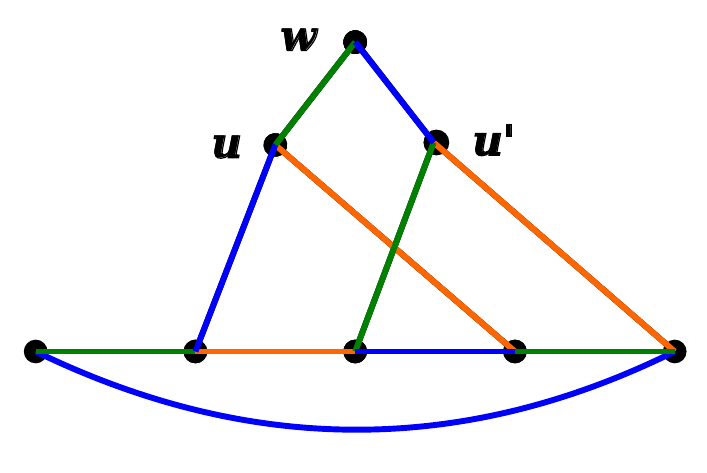}
    \caption{\label{pic:thm::s224_8}}
    \end{subfigure}
    \centering    
    \begin{subfigure}{0.3\textwidth}
    \centering  
    \includegraphics[width=0.8\linewidth]{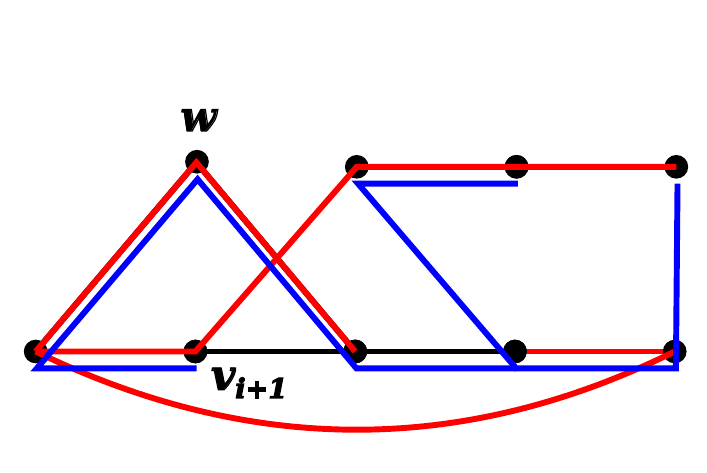} 
    \caption{}\label{pic:thm::s224_10}
    \end{subfigure}
    \caption{Forbidden subgraphs and coloring of neighborhood of $C_5$ in Case 2 of Theorem~\ref{thm::N113}.}
\end{figure} 

Thus, we can assume $N_2(Q)= \emptyset$. Let us consider $N(Q)$.

If $N(Q)$ contains two disjoint edges $xx'$ and $yy'$, then we can always find two vertices $v_i, v_{i+2} \in Q$ adjacent to $xx', yy'$ respectively and we have $S_{2,2,4}$. So $N(Q)$ is $2K_2$-free.  

Suppose $|N(Q)| \geq 4$. Then at least three vertices have only one neighbor on $Q$, so there are at least two edges in $N(Q)$ (since $\delta(G) > 1$). So, since $N(Q)$ is $2K_2$-free, it does contain $P_3$ and at least one isolated vertex $w$. Note that without loss of generality there exists $i$ such that $w$ is adjacent to $v_i, v_{i+2} \in Q$. Whether some end-vertex of $P_3$ is adjacent to $v_{i+1}$ or not, we can always find $S_{2,2,4}$ (see Figure \ref{pic:thm::s224_10}). 

Thus, $|N(Q)| \leq 3$ and $N(Q)$ is isomorphic to some subgraph of $P_3$.

Note that if $H$ does not contain $K^*_{3, 3}$ (see Figure \ref{pic:thm::s224_11}), then it is a subgraph of one of $3$-edge-colorable graphs depicted in Figures \ref{pic:thm::s224_12} and \ref{pic:thm::s224_13}. It is trivial if $|H| \leq 7$, so suppose $|H|=8$ and $H$ is maximal in the sense that adding another vertex or another edge would destroy one of properties of $H$ (triangle-freeness, subcubicidity, minimality with respect to Reductions). If $N(Q)$ contains an isolated vertex, then (since $K^*_{3, 3} \nleq H$) we can always add in $N(Q)$ an edge between the isolated vertex and one of the remaining ones, which contradicts the maximality of $H$. Thus, $H[N(Q)]$ is isomorphic to $P_3$. If both end-vertices of the path $P_3$ are of degree $1$ on the cycle, then there are no more than $3$ edges between $Q$ and $N(Q)$ and the graph is not maximal. Thus, one endvertex is adjacent to two vertices of $Q$, say $v_1$ and $v_3$. If the second endvertex is adjacent to $v_2$, $H'$ is isomorphic to the $3$-edge-colorable graph depicted in Figure \ref{pic:thm::s224_12}. Otherwise, from the triangle-freeness of $H'$, the second endvertex has degree $1$ on the cycle, and from the maximality of $H'$, the middle vertex is adjacent to $v_2$. Thus, $H'$ is isomorphic to the graph depicted in Figure \ref{pic:thm::s224_13}.

\begin{figure}[htb]
    \begin{subfigure}{0.3\textwidth}
    \centering  
    \includegraphics[width=0.8\linewidth]{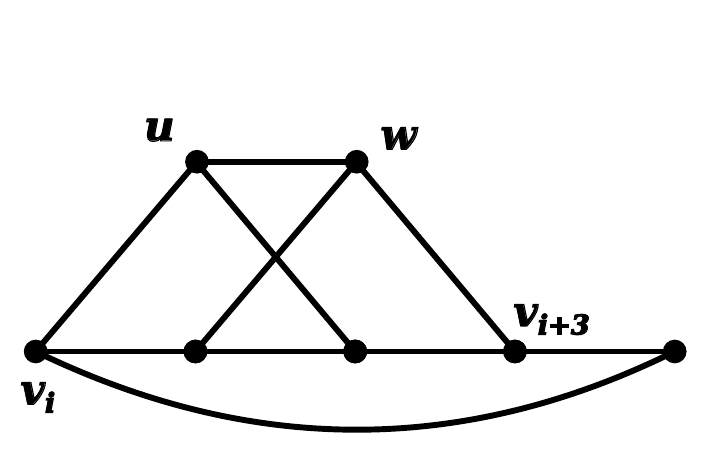}
    \caption{}\label{pic:thm::s224_11}
    \end{subfigure}
    \begin{subfigure}{0.3\textwidth}
    \centering  
    \includegraphics[width=0.8\linewidth]{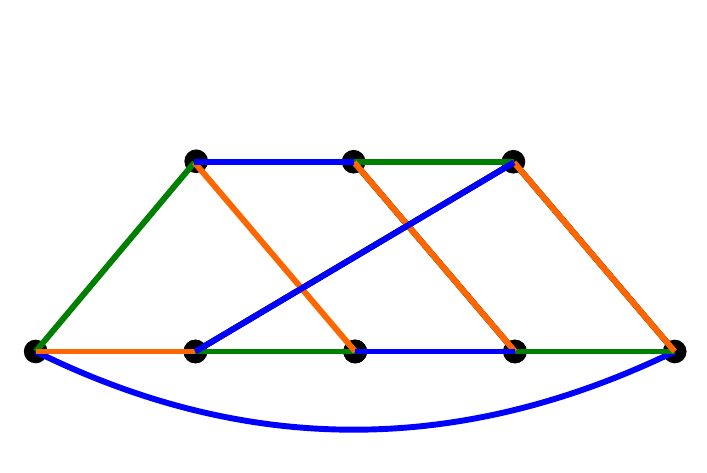}
    \caption{}\label{pic:thm::s224_12}
    \end{subfigure}
    \begin{subfigure}{0.3\textwidth}
    \centering  
    \includegraphics[width=0.8\linewidth]{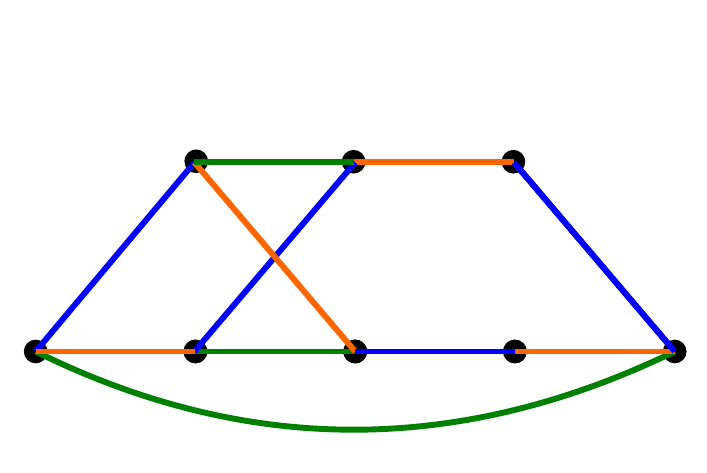}
    \caption{}\label{pic:thm::s224_13}
    \end{subfigure}
    \caption{Possible neighborhood of $C_5$ and its coloring in Case 2 of Theorem \ref{thm::N113}.}
\end{figure} 

\end{proof}

\subsection{The (claw, diamond, N$_{2,2,2}$)-free graphs without short cycles}

In this section we will prove that for the class of $(claw, diamond, N_{2,2,2}, C_5)$-free graphs we have only two families of exceptional graphs. It is the first part of analysis of the $(claw, diamond, N_{2,2,2})$-free class. In the next subsection, a similar theorem will be stated for  $(claw, diamond, N_{2,2,2})$-free graphs containing $C_5$.  
\begin{theorem}\label{thm::N222}
    Let $G$ be a connected, $(claw, diamond)$-free graph. If $G$ is also $(N_{2, 2, 2}, C_5)$-free, then
    \begin{enumerate}
        \item $G$ contains $K_4$ or
        \item $G$ contains $B_{9i +1}$ for some $i \geq 1$ or
        \item $G$ contains $B_{9i +7}$ for some $i \geq 1$ or
        \item $G$ is $3$-colorable. 
    \end{enumerate}
\end{theorem}

In fact, Theorem \ref{thm::N222} will follow as a corollary from the more specific theorem below.

\begin{theorem}\label{thm::N222lin}
    Let $H$ be a connected, subcubic triangle-free graph, minimal with respect to the Reductions 1-4. If $H$ is also $( S_{3,3,3}, C_5)$-free, then
     \begin{enumerate}
        \item $H$ is isomorphic to $D_{6i +1}$ for some $i \geq 1$ or
        \item $H$ is isomorphic to $D_{6i +5}$ for some $i \geq 1$ or
        \item $H$ is $3$-edge-colorable. 
    \end{enumerate}
\end{theorem}

\begin{proof}

    If $H$ is bipartite, then $\chi'(H)=\Delta(H)$, so let us assume otherwise and consider the shortest induced odd cycle $Q=v_1 \ldots v_p$. Since $H$ is $C_5$-free, we have $p \geq 7$. Of course, $N_3(Q) =\emptyset$, since otherwise we can easily find $S_{3,3,3}$. Suppose now $u\in N_2(Q)$ and $v \in N_1(Q)$ is a neighbor of $u$. But $u$ must have another neighbor $v' \neq v$ (by Reduction 1 and triangle-freeness of $H$) and in the graph $Q \cup {v, u, v'}$ we can easily find a star $S_{3,3,3}$. Thus, $N_2(Q) = \emptyset$. By the same argument we have that $N(Q)$ is $P_3$-free.

    Since $Q$ is a shortest cycle in $H$, then for every $w \in N(Q)$ we have $d_Q(w) \leq 2$. Moreover, if $d_Q(w)=2$, then $N_Q(w) = \{v_i, v_{i+2}\}$ for some $i$, and if $ww' \in E(H)$ for some $w'\in N(Q)$, then $N_Q(\{w, w'\}) \subset \{v_i, \ldots, v_{i+3}\}$ for some $i$. 

    Thus, the graph $H$ is a necklace of following gadgets (note that by Reductions 2 and 3 every gadget must contain at least $2$ vertices from $N(Q)$ and one of those vertices must have two neighbors on the cycle):

\begin{figure}[htb]
    \centering    
    \begin{subfigure}{0.22\textwidth}
    \centering  
    \includegraphics[width=1\linewidth]{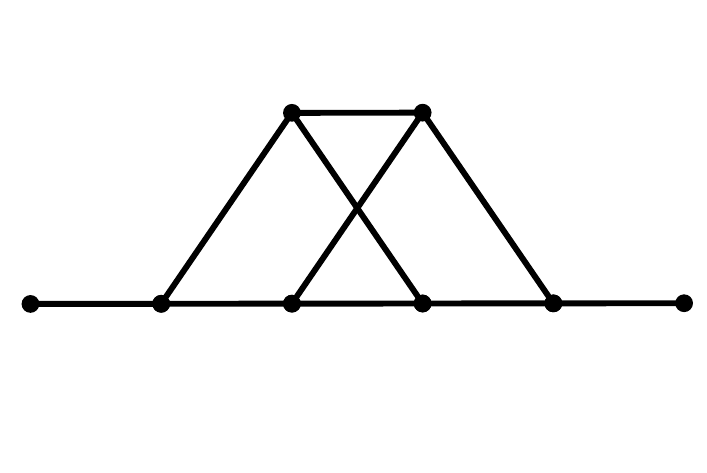} 
    \caption{\label{pic:claim_a}}
    \end{subfigure}
    \begin{subfigure}{0.22\textwidth}
    \centering  
    \includegraphics[width=1\linewidth]{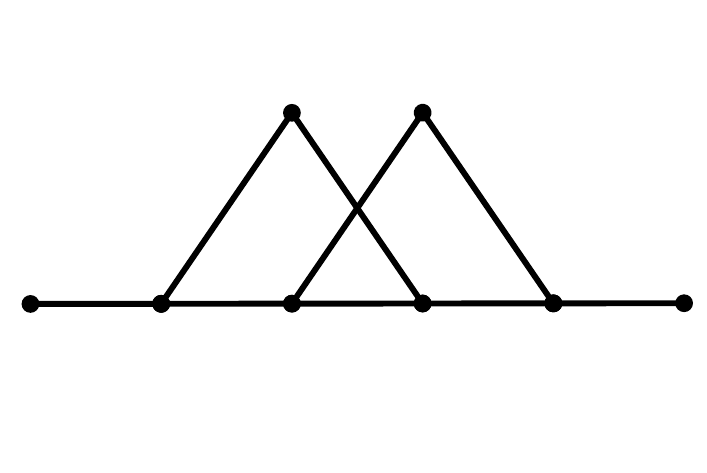}
    \caption{\label{pic:claim_b}}
    \end{subfigure}
    \begin{subfigure}{0.22\textwidth}
    \centering  
    \includegraphics[width=1\linewidth]{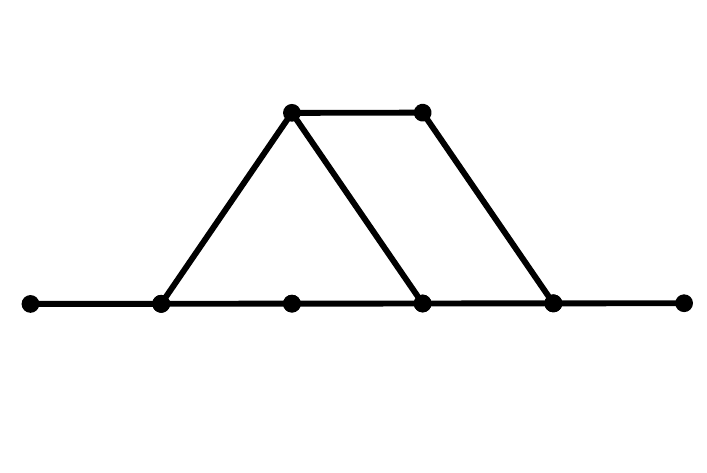}
    \caption{\label{pic:claim_c}}    
    \end{subfigure}
    \begin{subfigure}{0.22\textwidth}
    \centering  
    \includegraphics[width=1\linewidth]{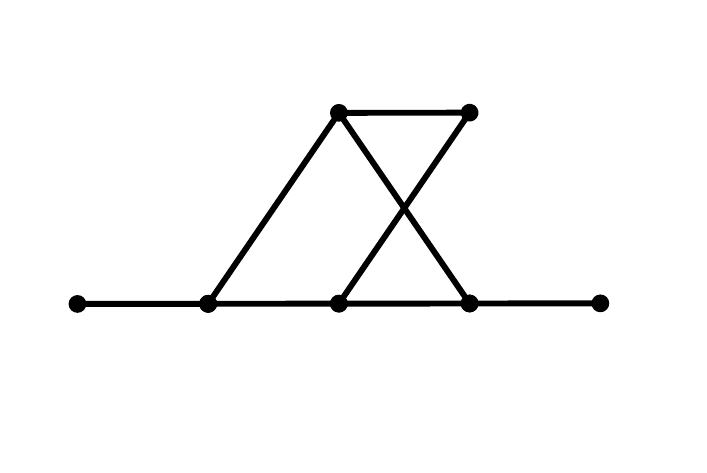}
    \caption{\label{pic:claim_d}}    
    \end{subfigure}
    \caption{Gadgets of graph $H$ in the proof of Theorem \ref{thm::N113}.}\label{pic:claim}
\end{figure}
    
    Note that gadgets $(b)$ and $(c)$ admit coloring with the same colors on the end-edges and different as well. Gadget $(a)$ admits only colorings with the same colors on the end-edges and $(d)$ - only with different colors. Thus, if $H$ is non-$3$-edge-colourable, it can contain only $(a)$ gadgets and at most one $(d)$ gadget. Moreover, the cycle can contain no more than $1$ vertex not belonging to any gadget, and if it contains such a vertex, then it does not contain a gadget $(d)$. This leaves us with two possibilities only - a necklace of gadgets $(a)$ with one edge subdivided (family $D_{6i+1}$) and a necklace of gadgets $(a)$ and one gadget $(d)$ (family $D_{6i+5}$). 

    This finishes the proof of Theorem \ref{thm::N222lin}.
    
\end{proof}

\subsection{The (claw, diamond, N$_{2,2,2}$)-free graphs containing $C_5$}

 In the main proof of this subsection we will be using several technical edge-coloring lemmas. 
 
 The first lemma provides an algorithm coloring edges of a Hamiltonian graphs satisfying at least one of additional conditions.  
\begin{Lemma}\label{algorithm}
    Let $H$ be a connected, Hamiltonian, subcubic graph and let  $v_1 \ldots v_{2p+1}$ be the Hamiltonian cycle. If 
    \begin{enumerate}
        \item there is an edge $v_iv_j \in E(G)$  such that $|i-j|>1 $ and $d(v_{i-1})=d(v_{j-1})=2$ or
        \item there is an edge $v_iv_j \in E(G)$ such that $|i-j|>1 $ and $d(v_{i+1})=d(v_{j+1})=2$ or
        \item there is $i \in [2p+1]$ such that $d(v_{i-1})=d(v_{i+1})=2$, 
    \end{enumerate}
    then $\chi'(G)=3$.
\end{Lemma}
\begin{proof}
\textbf{Cases 1. and 2.} Without loss of generality we have $d(v_{i-1})=d(v_{j-1})=2$ and $j=i+2n$ (the other situations are symmetric). We will start with precoloring the graph $H$: we color with red all the chords and the edge $v_{i-1}v_i$, and the rest of the cycle with green and blue alternately, starting from $v_iv_{i+1}$. The only color conflict is caused in $v_i$ by chord $v_iv_j$. So we change colors on the red/blue-path starting from this chord. Note that this path ends in $v_{j-1}$, because $j=i+2n$.

\textbf{Case 3.} We start with precoloring as in the previous case. If there is no chord in the vertex $v_i$, we are done. Otherwise we change colors on the red/blue-path starting from this chord. If it ends in some vertex $v\neq v_{i-1}$, we are done also. But if it ends in vertex $v_{i-1}$, we still have color conflict, because edges $v_{i-2}v_{i-1}$ and $v_{i-1}v_i$ are both colored with red. Then we recolor $v_{i-1}v_i$ with green and $v_{i}v_{i+1}$ with red, obtaining a proper coloring. 

\begin{figure}[htb]
    \centering    
    \begin{subfigure}{0.4\textwidth}
    \centering  
    \includegraphics[width=1\linewidth]{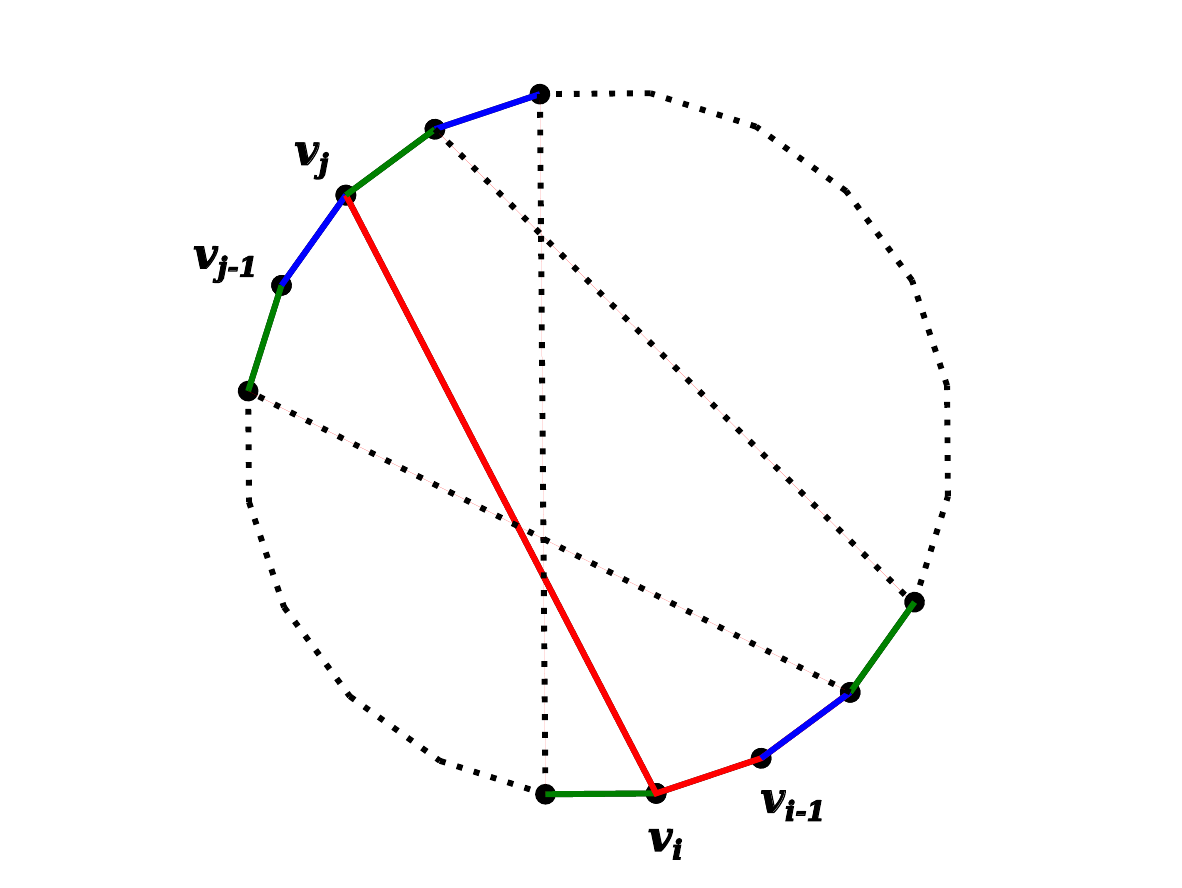} 
    \caption{Precoloring in Cases 1. and 2.\label{pic:alg1}}
    \end{subfigure}
    \begin{subfigure}{0.4\textwidth}
    \centering  
    \includegraphics[width=1\linewidth]{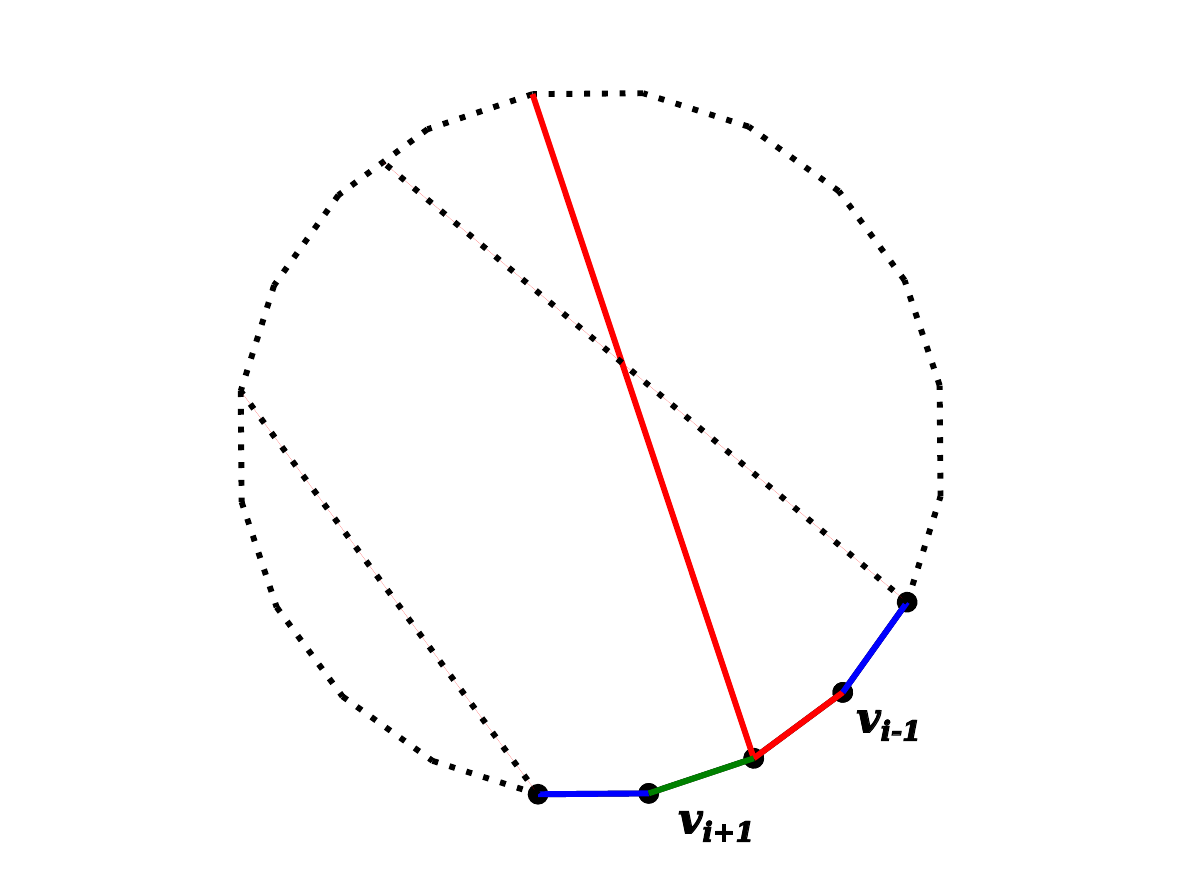}
    \caption{Precoloring in Case 3.\label{pic:alg2}}
    \end{subfigure}
    \caption{Precoloring used by the algorithm in Lemma \ref{algorithm}.}
\end{figure}

\end{proof}

The next lemma allows us to color the edges of a near-Hamiltonian graphs.

\begin{Lemma}\label{algorithm2}
    Let $H$ be a connected, subcubic, triangle-free graph, and let $v \in (G)$ be such a vertex that $d(v)=2$ and $G-v$ has Hamiltonian cycle $v_1\ldots v_{2p}$. Let $v_i, v_j$, where $|i-j| \geq 2$, be the neighbors of $v$ on the cycle. If
    \begin{enumerate}
        \item one of the vertices $v_{i-1}, v_{i+1}, v_{j-1}, v_{j+1}$ is of degree $2$ or
        \item there is a chord $v_{i-1}v_l$ such that either $v_{l+1}$ is of degree $2$, or such that $v_{l-1}$ is of degree $2$ and $|i-l| = 2n$  (for some $n \in \mathbb{N}$) or
        \item there is a chord $v_{i+1}v_l$ such that either $v_{l-1}$ is of degree $2$, or such that $v_{l+1}$ is of degree $2$ and $|i-l| = 2n$ (for some $n \in \mathbb{N}$),
    \end{enumerate}
    then $\chi'(G)=3$.
\end{Lemma}
\begin{proof}

\textbf{Case 1.}
Suppose $v_{i+1}$ is of degree $2$. We color $v_iv_{i+1}$ with red and the rest of cycle alternately with green and blue. Then we can color $vv_j$ and the chords of the cycle with red and $vv_i$ with blue without color conflict (see Figure \ref{pic:alg3}). In other cases the coloring is analogous.

\textbf{Cases 2 and 3.}
Suppose there exists a chord $v_{i+1}v_l$ and $v_{l-1}$ is of degree $2$. We color $v_iv_{i+1}, v_lv_{l-1}$ with red and $v_{i+1}v_l$ with blue. Then we can - without color conflict - color the rest of the cycle alternately with blue and green, $vv_j$ and the rest of chords with red, and $vv_i$ with remaining color. In other cases the coloring is analogous.

\begin{figure}[htb]
    \centering    
    \begin{subfigure}{0.4\textwidth}
    \centering  
    \includegraphics[width=1\linewidth]{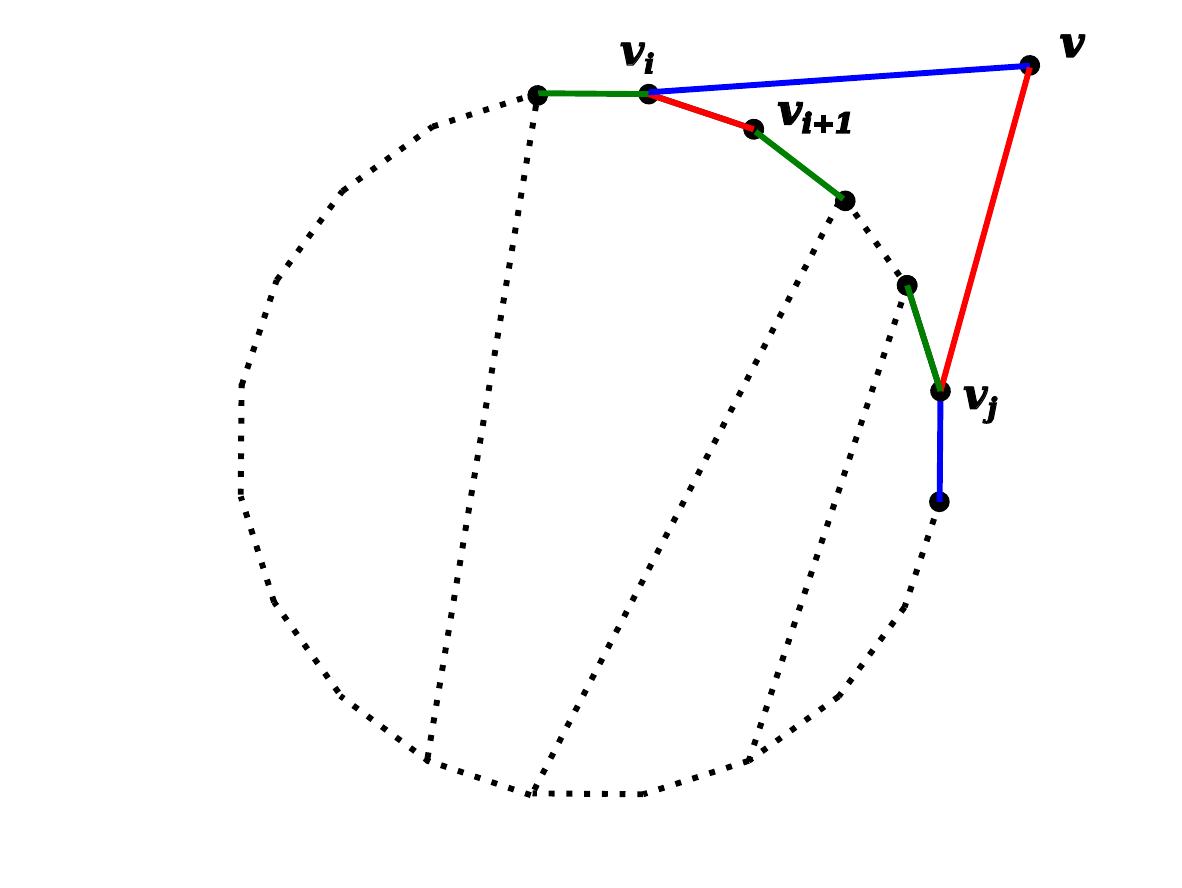} 
    \caption{Case 1.\label{pic:alg3}}
    \end{subfigure}
    \begin{subfigure}{0.4\textwidth}
    \centering  
    \includegraphics[width=1\linewidth]{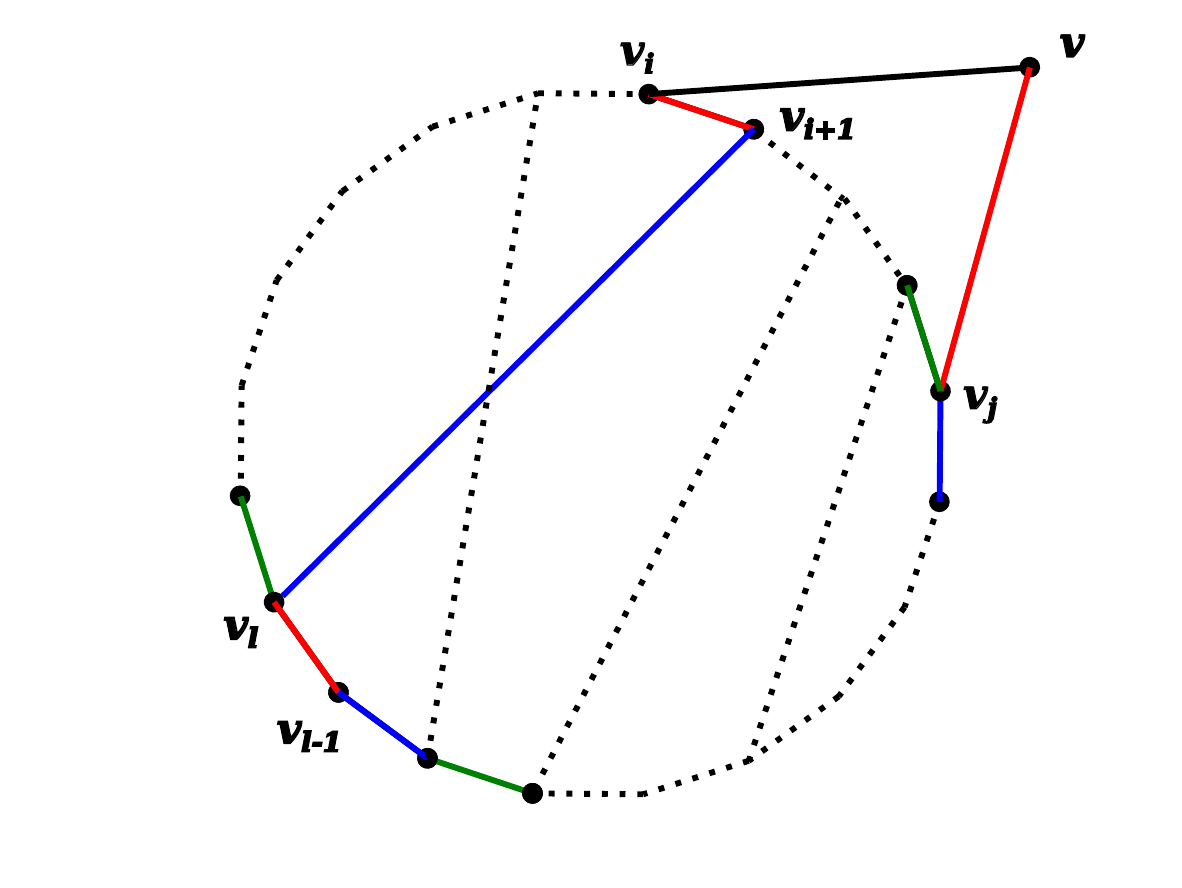}
    \caption{Cases 2 and 3.\label{pic:alg4}}
    \end{subfigure}
    \caption{Coloring in the proof of Lemma \ref{algorithm2}.}
\end{figure}

\end{proof}

In the case considered in this Section, the set of non-trivial exceptional graphs is more varied. It  contains the graph $B_{10}$, a line graph of $P^*$, and line graphs of all subcubic $9$-cycles with four chords. It is easy to observe that all graphs of the last group are overfull.

As before, the alternative stated in the theorem can be read as necessary and sufficient conditions for $3$-colorability of a $(claw, diamond, N_{2, 2, 2})$-free graph containing a cycle $C_5$. Here, the graph $P^*$ is the Petersen graph with one vertex deleted.

\begin{theorem}\label{thm::N222C5}
      Let $G$ be a connected, $(claw, diamond, N_{2, 2, 2})$-free graph. If $G$ contains an induced $C_5$, then
    \begin{enumerate}
        \item $G$ contains $K_4$ or
        \item $G$ contains $B_{10}$ or
        \item $G$ contains $L(P^*)$ or
        \item $G$ contains a line graph of an overfull graph $H$ consisting of an odd cycle $C_9$ with $4$ chords or
        \item $G$ is $3$-colorable.        
    \end{enumerate}
\end{theorem}

Theorem \ref{thm::N222C5} will follow as a corollary from Theorem \ref{thm::N222C5lin}.

\begin{theorem}\label{thm::N222C5lin}
      Let $H$ be a connected, triangle-free subcubic graph, minimal with respect to Reductions 1-4. If $H$ is also $S_{3, 3, 3}$-free and $H$ contains an induced $C_5$, then
    \begin{enumerate}
        \item $H$ is isomorphic to $D_7$ or
        \item $H$ is isomorphic to $P^*$ or
        \item $H$ is an overfull graph consisting of an odd cycle $C_9$ with $4$ chords or
        \item $H$ is $3$-edge-colorable.        
    \end{enumerate}
\end{theorem}

\begin{proof}
By assumption, $H$ contains a $5$-cycle  $Q=v_1v_2 v_3 v_4 v_5$. By Reduction 4 every component of $H - Q$ is connected to $Q$ with at least $2$ edges. Thus, we have at most $2$ components. 

Suppose now one of those components is an isolated vertex $\{w\}$. Then there exists such an index $i$ that the vertex $w$ is adjacent to $v_{i-1}$ and to $v_{i+1}$. 
But then by the Reduction 3 there must be a vertex $u$ outside the cycle, adjacent to $v_i$, and without loss of generality we can take the $5$-cycle $Q\setminus \{v_i\} \cup \{w\}$ and consider a component consisting of at least $v_i$ and $u$. So we will assume that any component has at least one edge.

Note that to find an $S_{3,3,3}$ it is sufficient to find one of the subgraphs depicted in Figure~\ref{pic:thm::N222C5}.

\begin{figure}[htb]
    \centering    
    \begin{subfigure}{0.3\textwidth}
    \centering  
    \includegraphics[width=1\linewidth]{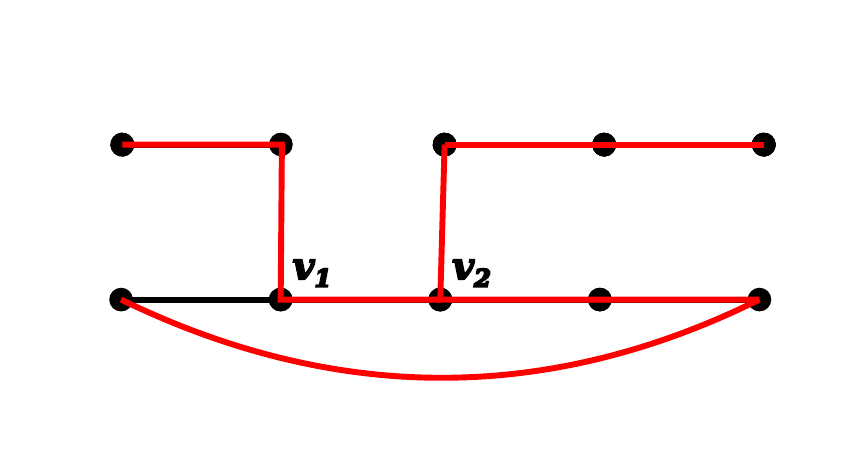} 
    \caption{}\label{pic:thm::N222C5-pic1}
    \end{subfigure}
    \begin{subfigure}{0.3\textwidth}
    \centering  
    \includegraphics[width=1\linewidth]{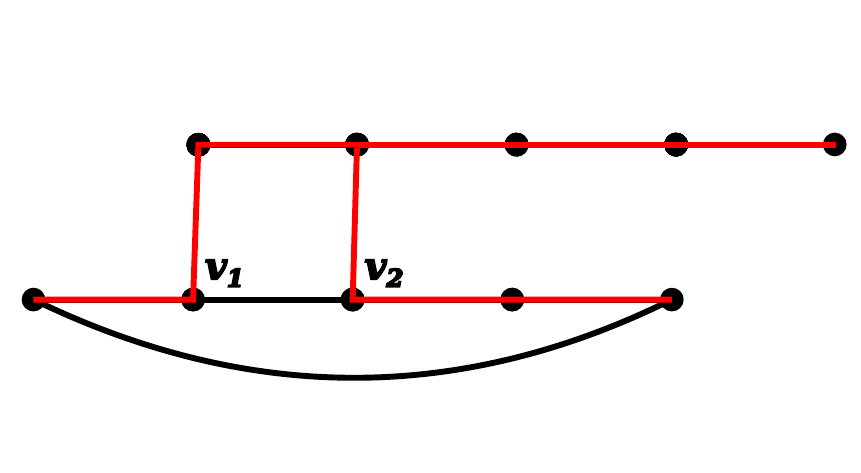}
    \caption{}\label{pic:thm::N222C5-pic2}
    \end{subfigure}
    \begin{subfigure}{0.3\textwidth}
    \centering  
    \includegraphics[width=1\linewidth]{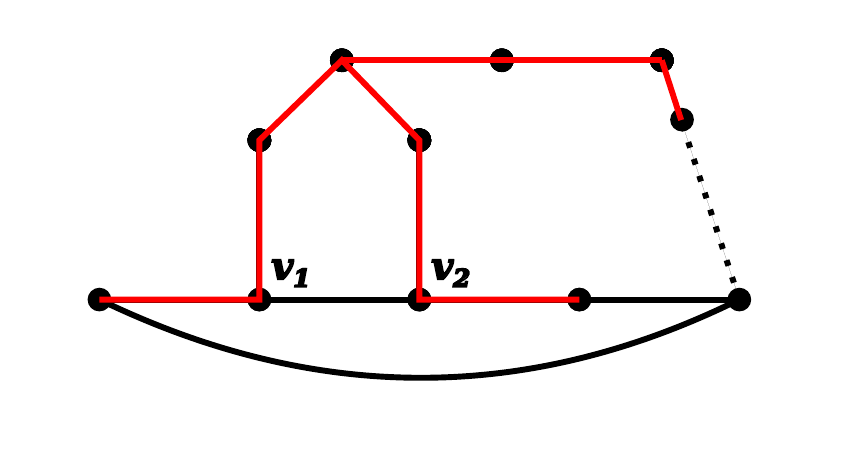}
    \caption{}\label{pic:thm::N222C5-pic3}
    \end{subfigure}
    \caption{Subgraphs occurring in the graphs considered in Case 1 of the proof of Theorem \ref{thm::N222C5}. }\label{pic:thm::N222C5}
\end{figure} 

\textbf{Case 1.} $|V(H)| \geq 10$. 

Recall that by Reductions each component of $H - Q$ is connected to $Q$ with at least $2$ edges and there are no vertices of degree 1. If we have $2$ components, say $C_1, C_2$, such that $|C_1| \le |C_2|$, then by the assumptions $2 \le |C_1|, 3 \le |C_2|$, and we can always find a $S_{3, 3, 3}$ (see Figure \ref{pic:thm::N222C5-pic1}). Thus, the graph $H-Q$ is connected. 

Let us consider the maximum degree $\Delta(H - Q)$. If we have $\Delta(H-Q) = 2$ and $H-Q$ is a cycle, then we always can find a subgraph \ref{pic:thm::N222C5-pic1}. If $\Delta(H-Q) = 2$ and $G-Q$ is a path, we always have one of subgraphs \ref{pic:thm::N222C5-pic1} or \ref{pic:thm::N222C5-pic2}, except when $H-Q$ is a path of length $5$ and $v_1, v_2$ are adjacent to its second and fourth vertex. But then by Reduction 1 both end-vertices of the path must be adjacent to the cycle and we can find another $i$ such that $v_i, v_{i+1}$ are adjacent to the path (so we can find a subgraph \ref{pic:thm::N222C5-pic1}).

If $\Delta(H-Q) = 3$, then there exists a spanning tree $T$ of $H-Q$ containing a star $S_{1, 1, 2}$. Let the vertex of degree $3$ of the star be a root of the tree $T$. Note that if we join two consecutive vertices $v_1, v_2$ with this spanning tree, we always obtain a subgraph \ref{pic:thm::N222C5-pic1}, except when $v_1, v_2$ are joined with two leaves of the same parent or when they are joined with a leaf and its parent.

Suppose $v_1, v_2$ are adjacent to two leaves $u_1, u_2$ of the same parent $u$ in $T$. If $u_1$ or $u_2$ is not of degree $1$ in $H-Q$, then we always can find a subgraph \ref{pic:thm::N222C5-pic1}, since from triangle-freeness of $H$ the vertex $u$ must have a parent (or another child, if $u$ is a root) non-adjacent to $u_1, u_2$.
If both $u_1$, $u_2$ are of degree $1$ in $H-Q$, then there always exist a path $P = u v v'$ starting from $u$, such that $u_1, u_2 \notin P$. Then whether $v'$ is adjacent to $u_1, u_2$, to the cycle or to another vertex of $H-Q$, we can always find one of the subgraphs \ref{pic:thm::N222C5-pic1} or \ref{pic:thm::N222C5-pic3}.

Suppose now $v_1, v_2$ are adjacent to the leaf $u$ and its parent $u'$. Then either we can find a subgraph \ref{pic:thm::N222C5-pic2} or $T = S_{1, 1, 2}$, where $u$ and $u'$ belong to the only one subdivided edge of the star. If one of the other leaves of $T$ is adjacent to $u$, we have a subgraph \ref{pic:thm::N222C5-pic1}. Otherwise the other leaves are adjacent to  $Q$ and we can find another $i$ such that $v_i$, $v_{i+1}$ are end-vertices of disjoint $P_3$ and $P_4$ as depicted in \ref{pic:thm::N222C5-pic1}.

\textbf{Case 2.} $|V(H)| = 9$

Note that $H - Q$ is a path $P_4$, a cycle $C_4$ or a star $K_{1,3}$. If $H-Q$ is a cycle, then by Reductions 2 and 3 there are at least $3$ edges between $Q$ and $H-Q$. Thus, we can always find a $9$-cycle. By Lemma \ref{algorithm} we know that $H$ is either an overfull graph of $13$ edges or is a $P^*$, or is $3$-edge-colorable. 

If $H-Q$ is a star $K_{1, 3}$, then by Reduction 1 all three leaves of $H-Q$ must be adjacent to $Q$. Since $\Delta(H) \leq 3$, at least one of the leaves has only one neighbor on the cycle. Whether there are $2$, $1$ or $0$ leaves with degree $2$ on the cycle, we can always find a $8$-cycle $Q'$ formed by $Q$, two leaves and the center of the star, such that the remaining leaf has degree $2$ on $Q'$. Then we can use Lemma \ref{algorithm2}. Since the length of the cycle is $8$, one of cases $1, 2, 3$ of Lemma  \ref{algorithm2} always occurs, so the graph is $3$-edge-colorable.  

If $H-Q$ is a path $P_4 = u_1 \dots u_4$, then by Reduction 1 both end-vertices must be adjacent to $Q$. If they are adjacent to two consecutive vertices of the cycle $Q$, then we can easily find $C_9$ and we use Lemma \ref{algorithm} as before, so assume otherwise - without loss of generality  $u_1v_1, u_4v_3 \in E(H)$. Since $H$ is triangle-free, we have $d(u_1) = d(u_4) = 2$. By Reduction 2 both vertices $u_2, u_3$ must be adjacent to the cycle $Q$. If one of them is adjacent to $v_2$, then we can find an $8$-cycle containing $Q$ and three vertices of $P_4$, where remaining end-vertex is of degree 2; if not, there is an $8$-cycle containing $P_4$ and four vertices of $Q$, where remaining vertex, $v_2$, is of degree 2. Then we use Lemma \ref{algorithm2} as before. 

\textbf{Case 3.} $|V(H)| \leq 8$

If $|V(H)| \leq 8$, then $H$ is $S_{2, 2, 4}$-free, so $L(H)$ is $N_{1,1, 3}$-free. Then the claim follows as a corollary from Theorem \ref{thm::N113}.
    
\end{proof}

\section{Further work}

{Recall, that Lozin and Purcell \cite{LozPur} have shown that the $3$-colorability problem remains $NP$-complete even for the 
class of $(claw, diamond, K_4, C_4, C_5, \ldots C_k)$-free graphs for any integer $k \geq 4.$}

Let us note that since obviously Reductions 1-4 can be applied in polynomial time, the following corollary results from Theorems \ref{thm::N222lin} and \ref{thm::N222C5lin}. 

\begin{corollary}
    Let $H$ be a connected, triangle-free subcubic graph. If $H$ is also $S_{3,3,3}$-free, then there is a polynomial algorithm deciding whether or not $H$ is $3$-edge-colorable.
\end{corollary}

Every graph different from $K_3$ and $K_{1,3}$ is uniquely determined by its line graph and can be restored from this line graph in polynomial time, so we can also state the following. 

\begin{corollary}
    Let $G$ be a connected, $(claw, diamond)$-free graph. If $G$ is also $N_{2, 2, 2}$-free, then there is a polynomial algorithm deciding whether or not $G$ is $3$-colorable.
\end{corollary}

{Moreover, we show that there are only finitely many non $3$-colorable $N_{1, 2, k}$-graphs for any $k \geq 2$, but there exist infinitely many non $3$-colorable $N_{i, j, k}$-graphs for any  $2 \leq i \leq j \leq k.$}

\begin{theorem}
    Let $G$ be a $(claw, diamond, N_{1,2,k})$-free graph with an integer $k \geq 2$. If $n > (k+4)(2^k+1)+1$, then 
    \begin{enumerate}
		\item $G$ contains $K_4$ or
				\item $G$ is $3$-colorable. 
	\end{enumerate}
\end{theorem}

\begin{proof}
Suppose $G$ contains no $K_4.$ Following previous discussions and reductions we may assume $3 \leq \delta(G) \leq \Delta(G) = 4.$
If $G$ is perfect, then $\chi(G) = \omega(G) \leq 3$ and so $G$ is $3$-colorable. Hence we may assume 
that $G$ is not perfect. Since $G$ is diamond-free, $G$ does not contain an induced odd antihole of size at least seven. 
Hence we may further assume that $G$ contains an odd hole $Q = v_1v_2 \ldots v_p$ of order $p \geq 5.$ Then $N^1(Q) \neq \emptyset$ 
since $\delta(G) \geq 3.$ If $w \in N^1(Q)$ then there are four consecutive vertices $v_i, v_{i+1}, v_{i+2}, v_{i+3}$ 
with $v_{i+1} v_{i+2} \in E(G)$ and $v_i v_{i+3} \notin E(G),$ since $G$ is $(claw,diamond)$-free. If 
$N_Q(w) = \{v_i, v_{i+1}, v_j, v_{j+1}\}$ for two integers with $1 \leq i < j \leq p$ and $|i-j| \geq 3,$ then there is 
a shorter odd hole, a contradiction. Therefore, $d_Q(w) = 2$ for every vertex $w \in N^1(Q).$ Since $d_{N^1}(w) \leq 2$ 
for every vertex $w \in N^1(Q)$ we obtain 
$$|N^1(Q)| \leq |Q| = p.$$ 
If $N_Q(w) = \{v_i, v_{i+1}\},$ then we label $w$ with $w_i$.

Moreover, since $G$ is $(claw, K_4)$-free, 
$$|N^{i+1}(Q)| \leq 2|N^i{Q}|$$
for all $i \geq 1.$ 

{Next we will show that}

{\begin{Claim}
$N^{k+1}(Q) = \emptyset$.
\end{Claim}}

\begin{proof}
Suppose $N^{k+1}(Q) \neq \emptyset.$ Assume that there is a path $w_1u_1u_2 \ldots u_{k+1}$ with $u_i \in N^{i+1}$ for $1 \leq i \leq k.$
Suppose first that $w_2$ exists, implying $w_2v_2, w_2v_3 \in E(G).$ Then $w_1w_2 \notin E(G)$ since $G$ is diamond-free. Since $G$ is claw-free, $u_1w_2 \notin E(G).$ But then $\{v_{p-1}, v_p, v_1, v_2, w_1, w_2, u_1, \ldots, u_k\}$ induces a $N_{1,2,k},$ a contradiction. 
So, suppose next that $w_2$ (and so by symmetry $w_p$) does not exist. Then $w_3, w_{p-1}$ exist. If $w_1w_3, u_1w_3 \notin E(G),$ then  
$\{v_p, v_1, v_2, v_3, w_1, w_3, u_1, \ldots, u_k\}$ induce a $N_{1,2,k},$ a contradiction. Hence we may assume that $w_1w_3, u_1w_3 \in E(G)$ since $G$ is claw-free. Now observe that $d(w_1)=4.$ Repeating our arguments we conclude that $w_{p-1}w_1, u_1w_{p-1} \notin E(G).$ 
But then $\{v_p, v_1, v_2,$ $ v_3, w_{p-1}, w_1, u_1, \ldots, u_k\}$ induces a $N_{1,2,k},$ a contradiction. 
\end{proof}

Therefore, $$|V(G)| \leq p+p+ \sum_{i=1}^{k-1}2^i = 2p + p \cdot \frac{2^k-1}{2-1} = p(2^k+1.)$$
Since $n > (k+4)(2^k+1)+1$ we conclude that $p \geq k+5.$ Now, if $N^2(Q) \neq \emptyset,$ then there is an induced $N_{1,2,k},$ a contradiction. Hence, we may assume that $N^2(Q) = \emptyset.$

Since $\delta(G) \geq 3$ it follows that $N(w_i) \cap N^1(Q) \neq \emptyset.$ Since $Q$ is a shortest odd hole we deduce that $N(w_i) \cap N^1(Q) \subset \{w_{i-2}, w_{i+2}\}.$ 
If $N(w_i) \cap N^1(Q) = \{w_{i-2}, w_{i+2}\},$ then there is a claw, a contradiction. Hence, we may assume that $G[N^1(Q)]$ is an induced matching.

Now, we $3$-color $Q$ as follows: If $p \equiv 0 (mod \ 3)$ then we color $v_1, v_2, v_3, \ldots$ alternatingly by $1,2,3, \ldots.$ 
If $p \equiv 1 (mod \ 3)$ then we color $v_1, v_2, v_3, \ldots$ by $1,2,1,3$ and then alternatingly by $1,2,3, \ldots.$ 
If $p \equiv 2 (mod \ 3)$ then we color $v_1, v_2, v_3, \ldots$ by $1,2,1,3,1,2,1,3$ and then alternatingly by $1,2,3, \ldots.$  Next, observe that $\{c(v_i),c(v_{i+1})\} \neq \{c(v_{i+2}),c(v_{i+3})\}$ for $1 \leq i \leq p.$ Therefore, we can extend the $3$-coloring of 
$Q$ with two suitable colors to a $3$-coloring of $Q$ for every edge $w_iw_{i+2}.$ So, we obtain a $3$-coloring of the whole graph~$G.$
\end{proof}

Thus, we can state the following corollary.

\begin{corollary}
    Let $G$ be a $(claw, diamond, N_{1,2,k})$-free graph. Then there is a polynomial algorithm deciding whether or not $G$ is $3$-colorable.
\end{corollary}

Let us recall that the non-$3$-colorable, infinite families of graphs $B_{9i+1}$ and $B_{9i+7}$ are $(claw, diamond, N_{2, 2, 2})$-free. Thus, we can state the following.

\begin{corollary}
    Let $2 \leq i \leq j \leq k$. Then there exist infinitely many non-$3$-colorable, $(claw, diamond, N_{i, j, k})$-free graphs not containing the complete graph $K_4$.
\end{corollary}}

\begin{corollary}
    Let $G$ be a $(claw, diamond, N_{1,2,k})$-free graph with $2 \leq k$. Then there are only finitely many non-$3$-colorable graphs not containing $K_4$.
\end{corollary}


\bibliographystyle{abbrv}
\bibliography{main}

\end{document}